\documentclass[a4paper,12pt]{amsart}
\usepackage{amsmath,amssymb}
 \usepackage{graphicx}
 \usepackage{caption}
\newtheorem{thm}{Theorem}
\newtheorem{lemma}[thm]{Lemma}
\newtheorem{prop}[thm]{Proposition}
\newtheorem{cor}[thm]{Corollary}
\newtheorem{rem}[thm]{Remark}
\newtheorem{defi}[thm]{Definition}
\newcommand{\sign}{\text{\rm sign}}

\begin{document}

\title{A Generalization of a Classical Geometric Extremum Problem}

\author{Petar Kenderov}
\address{P. Kenderov\\Institute of Mathematics and Informatics\\Bulgarian Academy
of Sciences\\Acad. G. Bonchev 8, 1113 Sofia, Bulgaria}
\email{vorednek@gmail.com}

\author{Oleg Mushkarov}
\address{O. Mushkarov\\Institute of Mathematics and Informatics\\Bulgarian Academy
of Sciences\\Acad. G. Bonchev 8, 1113 Sofia, Bulgaria}

\email{muskarov@math.bas.bg}

\author{Nikolai Nikolov}
\address{N. Nikolov\\Institute of Mathematics and Informatics\\Bulgarian Academy
of Sciences\\Acad. G. Bonchev 8, 1113 Sofia, Bulgaria \vspace{1mm}
\newline Faculty of Information Sciences\\
State University of Library Studies and Information Technologies\\
Shipchenski prohod 69A, 1574 Sofia, Bulgaria}

\email{nik@math.bas.bg}

\thanks{The third-named author was partially supported by the Bulgarian National
Science Fund, Ministry of Education and Science of Bulgaria under contract KP-06-N82/6.}

\keywords{convex set, extremal problem, Philo' line, critical
point, convex polytope, concurrent lines}

\subjclass[2020]{51M16, 52A20, 52A40, 49K05, 49K10}

\begin{abstract}
Let $\partial \,\mathcal{C}$ be the boundary of a compact convex body $\mathcal{C}$ in $\mathbb{R}^n,\, n\geq 2$, and $O$ be an interior point of $\mathcal C$. Every straight line $l$ containing $O$  cuts from $\mathcal{C}$ a segment $[AB]$ with end-points on $\partial \,\mathcal{C}$. It is shown that if $[AB]$  is the shortest such segment, then $\partial \,\mathcal{C}$ is smooth at the points $A$ and $ B$ (i.e. at both of them there is only one supporting hyperplane for $\mathcal{C}$) and, something more, the normals to the unique supporting hyperplanes at the points $A$ and $B$ intersect at a point belonging to the hiperplane through $O$ which is orthogonal to $[AB]$.

If $\mathcal{C}$ is a smooth  compact convex body in          $\mathbb{R}^n,\, n\geq 2$, the above property holds also when $[AB]$ is the longest such segment. Similar results have place also when $O$ is outside the set $\mathcal{C}$.
The ``local versions'' of these results (when the length $|AB|$ of the segment $[AB]$ is locally maximal or locally minimal) also have a place. More specific results are obtained in the particular case when $\mathcal{C}$ is a convex polytope.

\end{abstract}

\maketitle

\section{Introduction}\label{si}

This article's starting point was the following well-known geometric extremum problem:

\smallskip

\emph{Given an angle in the plane and an interior point $O$, find
a line $l$ through $O$  which cuts from the angle a segment $[AB]$
of shortest length $|AB|$}.

\smallskip

It is known that this line exists and is uniquely determined by the following property  $(*)$ (see Figure 1):

\smallskip

\emph{ $(*)$ the foot  $E'$ of the perpendicular from the angle
vertex $E$  to $l$ belongs to $ [AB]$ and $|BE'| = |AO|$}

\begin{figure}[h]  \label{f1a}
\centering
\includegraphics[scale=0.8]{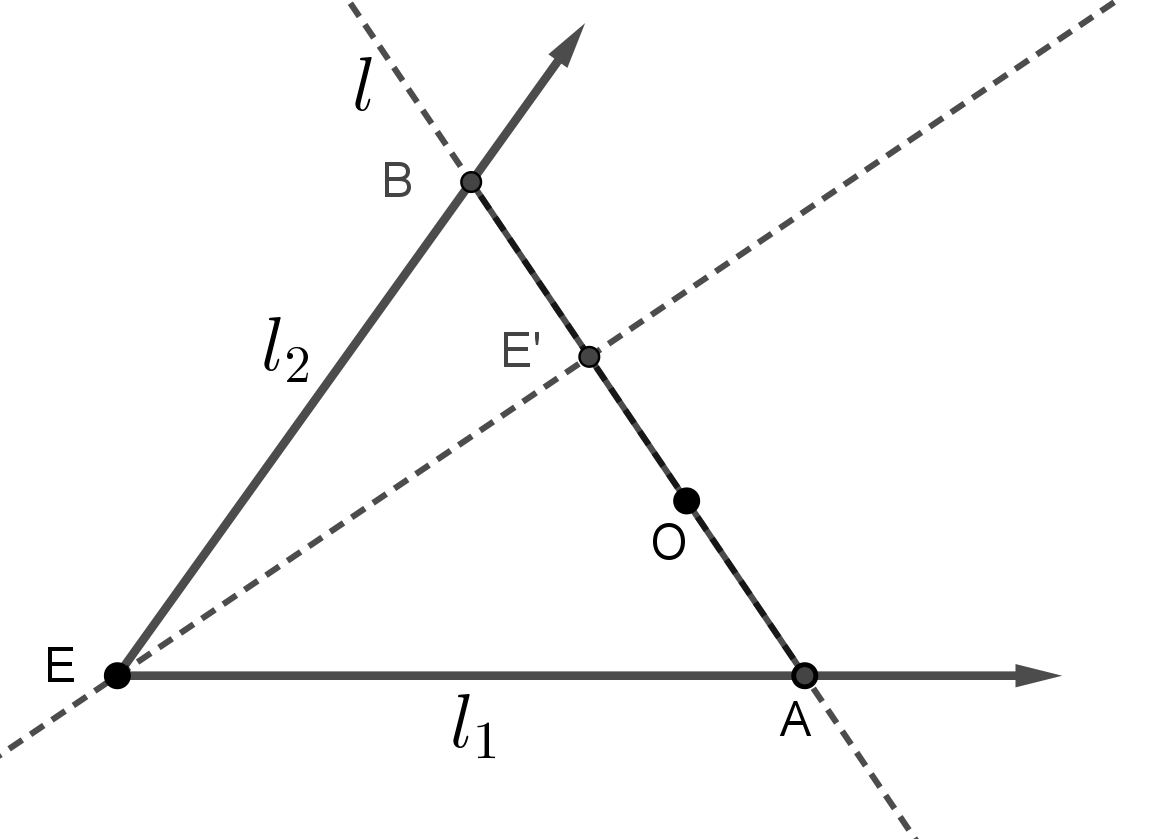}
\caption*{Figure 1}
\end{figure}


This characterization of the optimal line $l$ suggests a way for its construction, described in the articles of Eves \cite{EVES} and Coxeter and van de Craats  \cite{CoxCra}.   It is mentioned in \cite{CoxCra} that this procedure for constructing $l$ generalizes a construction ``published in 1647 by Gregory of St. Vincent'' for the particular case when the angle $\angle El_1l_2$ is right.

Given an angle $\angle El_1l_2$ and an interior point $O$ in it,
construct first the uniquely determined hyperbola $h$ passing
through $O$ and having as asymptotes the lines containing the arms
$l_1$ and $l_2$ of the angle. Next construct a circle $c$ with
$[EO]$ as diameter (see Figure 1a) and find the second
intersection point $E'$ of $c$ and $h$ (the first common point
being $O$ itself).

\begin{figure}[h]  \label{f1b}
\centering
\includegraphics[scale=0.8]{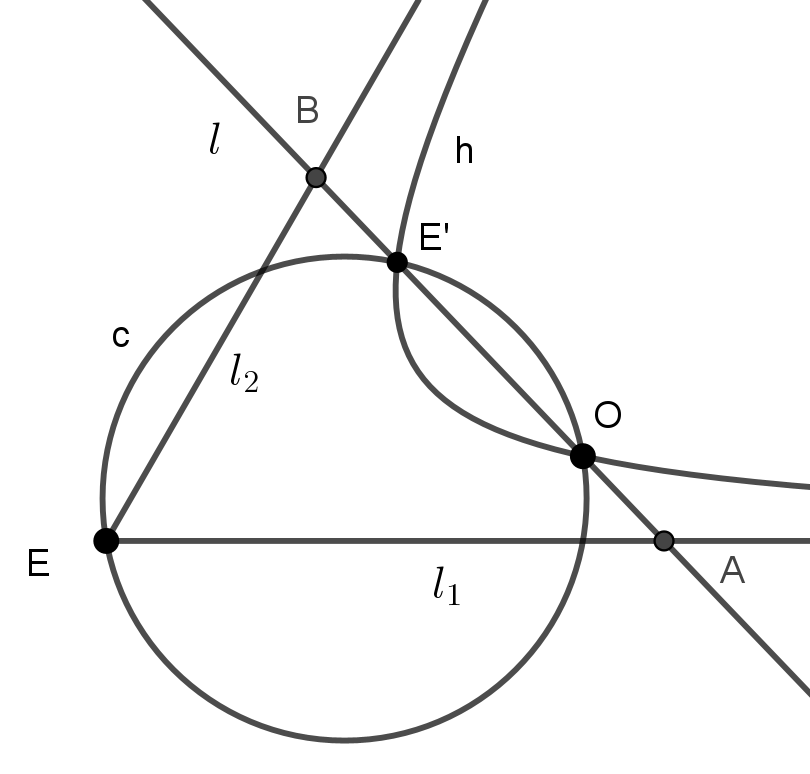}
\caption*{Figure 1a}
\end{figure}

The line $l$ determined by the points $O$ and $E'$ (and cutting
the arms of the angle at $A$ and $B$, respectively) has the
property $(*)$. Indeed, the equality $|BE'| = |OA|$ follows from
the well-known fact (see, for instance, Coxeter \cite{Cox61} 8.92
on p. 134) that any hyperbola secant (in particular, the secant $l
=OE'$) meets the asymptotes in such a way that $|OA|=|BE'|$. Note
also that the angle $\angle E E' O$ is right by construction.

\medskip

It is interesting that, for the particular case when the angle $\angle El_1l_2$ is right, the line $l$
 with the property $(*)$
appeared very early in connection with a completely different
problem. The famous Philo (or Philon) from Byzantium who lived in
the years 280 - 220 BC, used this line to provide one more
solution to the problem of cube doubling which occupied the minds
of the ancient philosophers for a long time (see \cite{Wiki1}).
For this reason, the line $l$ providing the solution to the
optimization problem formulated at the beginning of this article
is frequently called the ``line of Philon'' (or ``Philo's line'').

\medskip

With the advent of the infinitesimals and calculus, which are well
suited for solving optimization problems, the interest in the
above optimization problem increased significantly. Neuberg
\cite{NBRG} notes that the problem is mentioned in Newton's {\it
Opuscules, t.I, p.87}. In two papers (from the years 1795
\cite{ML} and 1811 \cite{ML1}) M. Lhuilier published the results
of his studies on this problem. In a footnote to the first page of
\cite{ML1} the editor of the journal mentions that the same
problem is considered (with different means) by M. Puissant. An
elementary proof of property $(*)$ was published in the fourth
edition of J. Casey's book  \cite{JCas} (p. 39, Proposition 18).
 In 1887 E. Neovius \cite{EN} generalized the problem by allowing the point $O$ to
 be outside the angle, and has found places  for the point $O$ such that the length $|AB|$ of the segment $[AB]$  has one local maximum and one local minimum when the line $AB$ rotates around the point $O$.
 From the more recent studies of this problem, we mention the articles of  H. Eves \cite{EVES},  of O. Bottema \cite{Bot} who used polar coordinates and trigonometry to simplify some of the results of Neovius, and the article of H.S.M. Coxeter and J. van De Craats \cite{CoxCra} mentioned above.

 \medskip

In the sequel, we focus our attention on another characterization
of the optimal line $l$, which is somehow less popular nowadays:

{\it The perpendicular to the line $l$ at the point $O$ and the
perpendiculars to the arms of the angle at the points $A$ and $B$, respectively  (dashed lines in Figure 1b), meet at a point}.

\begin{figure}[h]  \label{f1b}
\centering
\includegraphics[scale=0.8]{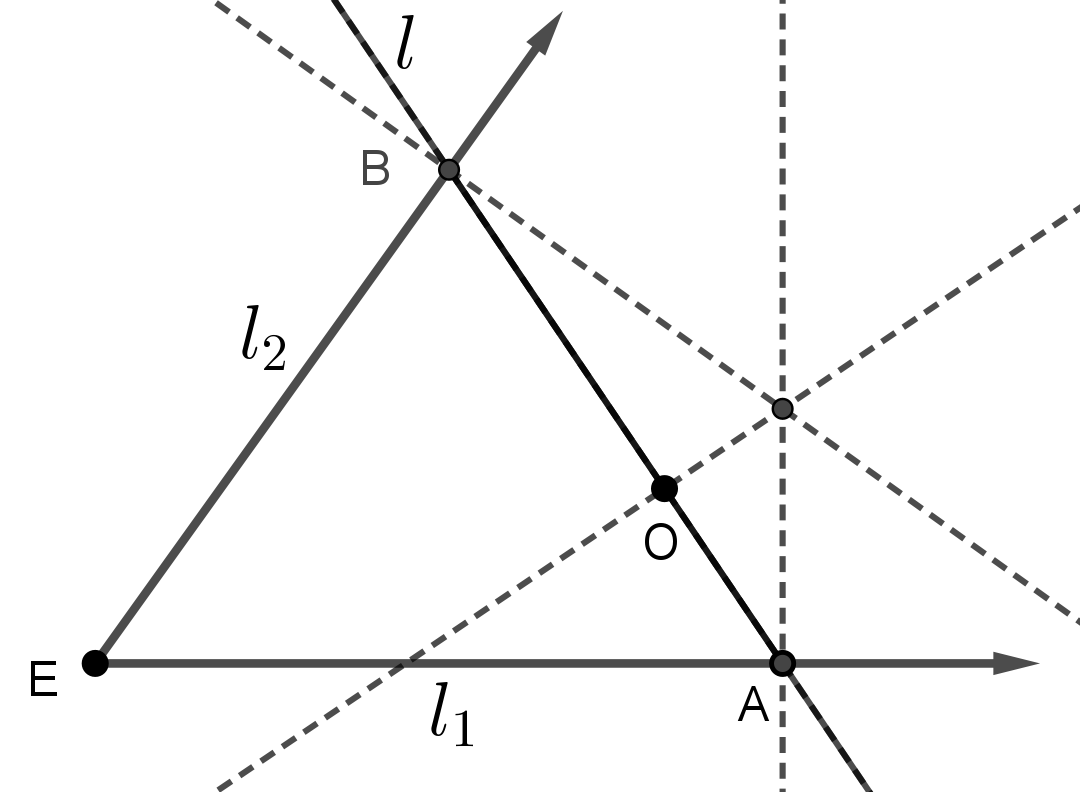}
\caption*{Figure 1b}
\end{figure}

A simple proof that this property is equivalent to the property (*) can be found in the article of Eves \cite{EVES}.
Without proof and as ``communicated by a friendly side'' Neovius mentions, at the end of his paper \cite{EN},
that this property remains valid also in the case when the point $O$ is outside the angle and the length $|AB|$ of the segment $[AB]$ cut from the angle by the line $l$ is locally maximal or locally minimal.

We call this property  \emph{Concurrent Perpendiculars Property} (CPP for short) and prefer to work with it because it does not involve the vertex $E$ of the angle and is valid (for trivial reasons) also when the lines $l_1$ and $l_2$ are parallel. Besides, this property is more amenable to generalizations, which is the main goal of this paper.

 \medskip

It turns out that CPP is a necessary (but not sufficient!) optimality condition in some other geometric extremal problems.
 For instance,  let $\mathcal{C}$ be any compact convex set in the plane $\mathbb{R}^2$ with a fixed interior point $O$.
 Consider all segments $[AB]$ cut from $\mathcal{C}$ by a line $l(\varphi)$ rotating around $O$ with rotation angle $\varphi$.
 The length $|AB|$ of the segment $[AB]$, considered as a function $\tilde{f}(\varphi)$ of the angle
  $\varphi$, may have several local minima (maxima) relative to $\varphi$.
  We prove in Section 2 that for every  local minimum $\varphi_*$ of $\tilde{f}(\varphi)$ with  corresponding segment $[A^*B^*]$ the following statements are valid:

a) $\mathcal{C}$ is \emph{smooth} at $A^*$ and $B^*$ (i.e. there is only one supporting line $l_1$ for $\mathcal{C}$ at $A^*$ and only one supporting line $l_2$ for $\mathcal{C}$ at $B^*$);

b) CPP holds for the lines $l_1, l_2,
A^*B^*$ (the line through $A^*$ and $B^*$)
and the points $A^*,
B^*, O ,$ respectively.

c) $\tilde{f}(\varphi)$ is differentiable at $\varphi_*$ and $\tilde{f}'(\varphi_*)=0$.

\smallskip
This result cannot be extended to local maxima of
$\tilde{f}(\varphi)$. For example, if $\mathcal{C}$ is a polygon
and $\varphi^*$ is a local maximum for $\tilde{f}(\varphi)$, then
at least one of the end-points $A^*$ and $B^*$ of the
corresponding segment $[A^*B^*]$ must be a vertex of
$\mathcal{C}$.

However, if $\mathcal{C}$ is a smooth compact convex set in the plane with interior point $O$, then b) and c) remain valid for local maxima as well.
 \medskip

In Section 2, we show that similar results are also valid when the point $O$ is outside the compact convex set $\mathcal{C} \subset \mathbb{R}^2$.

In Section 3, we show how the results from Section 2 for sets in $ \mathbb{R}^2$ imply results for sets in $\mathbb{R}^n, n >2$, which are more general than the results announced in the abstract of the paper.

\medskip

In Section 4 we consider the special case when $\mathcal{C}$ is a
convex polytope. and bring to the fore some specific for these
sets properties.
\medskip

The practical relevance of the problems considered here is discussed briefly in the Appendix to the article.
In what follows, we denote by
$\angle ABC$ both the elementary geometry angle and its measure in radians.

\medskip

\section{The case when the set $\mathcal{C}$ is in $\mathbb{R}^2 $. }

We start with a simple observation that reveals where CPP comes from.

Let $A$ be a point from a line $p$ and $O$ be a point outside $p$.
Denote by $O'$ the orthogonal projection of $O$ on $p$ and put
$\varphi= \angle AOO'$ (Figure 2). Depending on the position of
$A$, the angle $ \varphi$ varies between $-\pi /2$ and $\pi/2$. If
$A$ is to the left of $O'$,  $\varphi$ is negative. If $A$ is to
the right of $O'$, then $\varphi$ is positive. Denote by $A'$ the
intersection point of the line through A perpendicular to $p$ and
the line through $O$ perpendicular to $OA$ (the dashed lines
in Figure 2).

\begin{figure}[h]  \label{f2}
\centering
\includegraphics[scale=0.4]{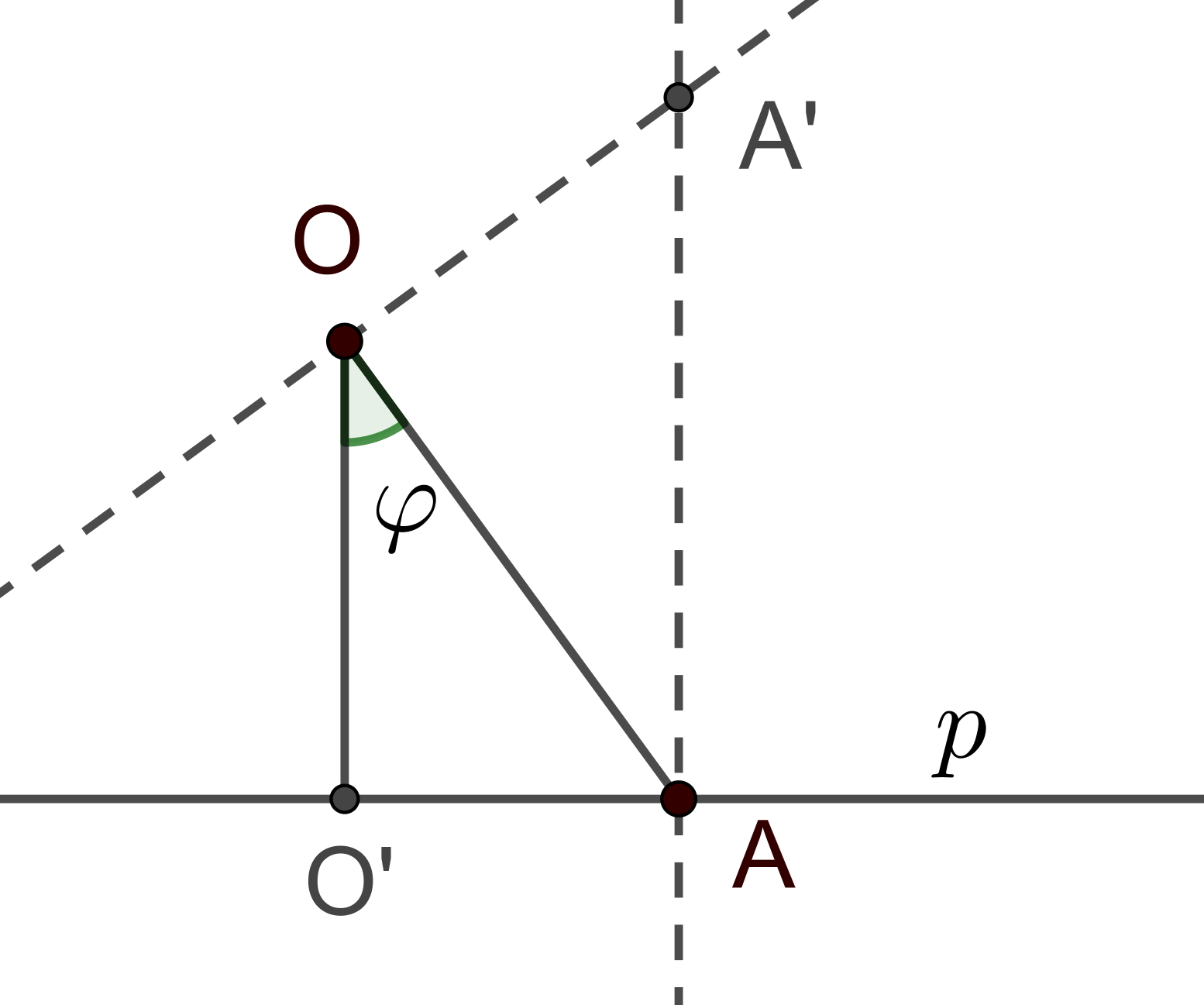}
\caption*{Figure 2}
\end{figure}

\begin{lemma} \label{L1}

The following relations hold for the first derivative $|OA|'_\varphi$ and for the second derivative $|OA|''_\varphi$  of $|OA|$ with respect to the angle $\varphi$ :

(a) $|OA|'_\varphi =|OA|\tan\varphi= \sign(\varphi)|OA'|$ ;

(b) $|OA|''_\varphi = |OA|(\tan^2\varphi + \cos^{-2}\varphi)$.
\end{lemma}
Here, as usual, $\sign(\varphi)= 1$ if $\varphi >0$, $\sign(\varphi)= - 1$ if $\varphi <0$, and
$\sign(\varphi)= 0$ if  $\varphi = 0$.

\begin{proof}
(a) Clearly, $|OA|=|OO'|\cos^{-1}\varphi$. Hence $|OA|'_\varphi
=|OO'|\cos^{-1}\varphi\tan\varphi =
|OA|\tan\varphi=\sign(\tan\varphi)|OA'|= \sign(\varphi)|OA'|$.

\noindent (b) Differentiating $|OA|'_\varphi = |OA|\tan \varphi$
we get $|OA|''_\varphi = |OA|'_\varphi \tan \varphi + |OA|
\cos^{-2}\varphi = |OA| (\tan^2 \varphi + \cos^{-2}\varphi)$.
\end{proof}

\bigskip

\subsection{Point $O$ is in the interior of $\mathcal{C}$}

We first prove the validity of CPP for angles and then derive this
property for more general sets $\mathcal{C}$. Consider an angle
$\angle E l_1l_2$ with vertex $E$,  arms $l_1$ and $l_2$, and
angle measure $\theta$, $0< \theta <\pi$, (Figure 3). Let $O$ be
an interior point for this angle and let $O'$ and $O''$ be the
orthogonal projections of $O$ onto the lines containing $l_1$ and
$l_2$ , respectively. Let $[AB]$ be the segment cut from this
angle by a line $l$ passing through $O$ where $A \in l_1$ and
$B\in l_2$. Put  $\varphi:= \angle AOO'$ if $O'$ belongs to the
ray $\overrightarrow{AE}$ and $\varphi:= -\angle AOO'$ otherwise.
Clearly, the segment $[AB]$ exists only if $ \theta - \pi/2 <
\varphi < \pi/2$. Note that the angle $\varphi$ determines the
line $l$ corresponding to it.  We use the notation $l =
l(\varphi)$ to reflect this fact.

Denote by $A'$  the intersection point of the line through $O$
perpendicular to $l(\varphi)$  and  the line perpendicular to
$l_1$ at $A$ (dashed lines in Figure 3).

Analogously, put $\psi:= \angle O''OB$ if $O''$ belongs to the ray
$\overrightarrow{BE}$ and $\psi:= -\angle O''OB$ otherwise. Denote
by $B'$ the intersection point of the line through $O$
perpendicular to $l(\varphi)$ and the line through $B$
perpendicular to $l_2$.

\begin{figure}[h]  \label{f3}
\centering
\includegraphics[scale=0.4]{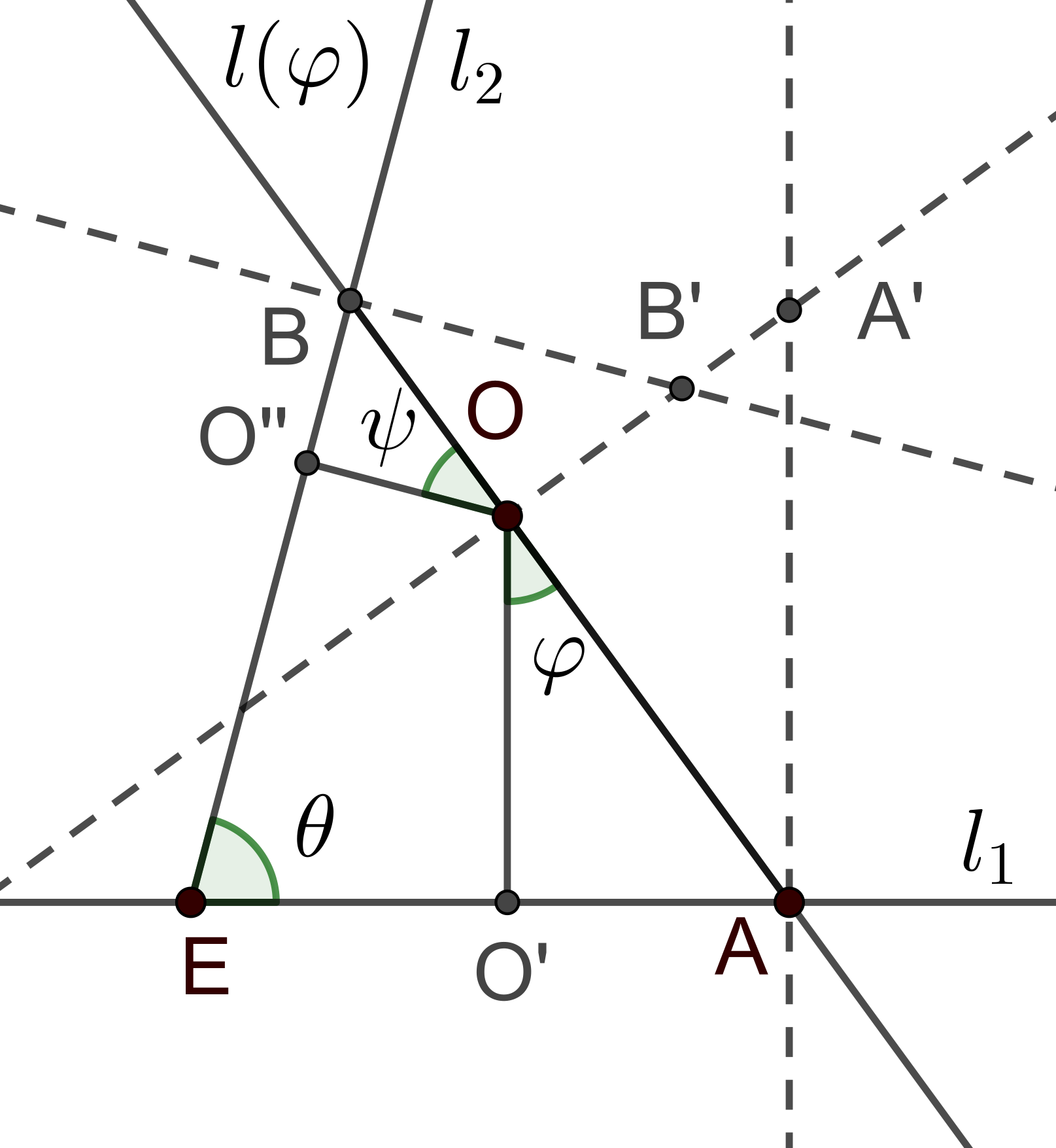}
\caption*{Figure 3}
\end{figure}

\begin{prop}\label{Rem1}
The function $f(\varphi):=|OA| + |OB|$ is differentiable and strictly convex.
For its derivative we have $f' (\varphi)= \sign(\varphi)|OA'| - \sign(\theta - \varphi)|OB'|$.
\end{prop}
\begin{proof}
As $ \theta - \pi/2 < \varphi < \pi/2$, we have $ \theta - \pi/2 < \theta -\varphi < \pi/2$ as well . It is not difficult to realize that, for every  possible $\varphi$, $\psi = \theta - \varphi$.
Evidently, $|OB|=|OO''|\cos^{-1}(\theta - \varphi)$ and we can calculate, as in Lemma 1, that

{\it (a')} $|OB|'_\varphi =- |OB| \tan (\theta - \varphi) =- \sign( \tan(\theta - \varphi))|OB'| = -\sign(\theta - \varphi)|OB'|$ and

{\it (b')} $|OB|''_\varphi = |OB|(\tan^2(\theta - \varphi) + \cos^{-2}(\theta -\varphi))$.

Strict convexity of $f$ follows from property (b) of Lemma 1 and (b'). The expression for the derivative follows from  property (a) of Lemma 1 and (a').
\end {proof}

\begin{prop} \label{mainprop}
(CPP for angles). Let $O$ be an interior point for an angle $\angle El_1l_2$ with angle measure $\theta, 0<\theta < \pi$. Consider all segments $[AB]$ cut from this angle by lines $\l(\varphi)$, $ \theta - \pi/2 < \varphi < \pi/2$.  For any angle  $\varphi_0$, $ \theta - \pi/2 < \varphi_0 < \pi/2$, the following statements are equivalent:

(i) $\varphi_0$ is a critical point of the function $f(\varphi)=
|AB|$ (i.e. $f'(\varphi_0) = 0)$.

(ii) The segment $[A_0B_0]$ corresponding to the line
$l(\varphi_0)$ is the shortest among all segments $[AB]$ cut from
the angle by lines passing through $O$.

(iii)  (CPP) The perpendiculars to $l_1$ at $A_0$, to $l_2$ at
$B_0$ and to the line $A_0B_0$ at $O$ meet at a point.

There exists only one  $\varphi_0$, $ \theta - \pi/2 < \varphi_0 < \pi/2$, for which one (and then all) of these three equivalent statements has a place.

\end{prop}

\begin{proof}
 We consider separately the cases when $0< \theta< \pi/2$, and when  $\pi/2  \leq  \theta < \pi$.

\smallskip

Assume  first that $0 < \theta< \pi/2$. This case is depicted in
Figure 3. For $\varphi \leq 0$, $f'(\varphi) = -|OA'| - |OB'| <0$.
If $\varphi \geq \theta$, then $f'(\varphi) = |OA'| + |OB'| >0$.
Strict convexity of the function $f(\varphi)$ implies that it has
a unique minimizer $\varphi_0$ which belongs to the open interval
$(0, \theta)$ and is characterized by the relation
$f'(\varphi_0)=0$.  This shows that (i) and (ii) are equivalent.

For $\varphi$ from the open interval $(0, \theta)$ we have $f'(\varphi) = |OA'| - |OB'|$. Note also that in this case $O \not \in [A'B']$.
Hence, $f'(\varphi) = |OA'| - |OB'| =0$ if and only if $A'=B'$. This means that (ii) and (iii) are equivalent which completes the proof of Proposition \ref{mainprop}
for the case when $0< \theta< \pi/2$.

\smallskip

The case $\pi/2 \leq \theta < \pi$ is even simpler. From $\varphi
> \theta - \pi/2 \geq 0$ and  $\theta - \varphi > \theta - \pi/2
\geq 0$ we derive that $f'(\varphi) = |OA'| - |OB'|$. The strictly
convex function $f(\varphi)$  tends to $+\infty$ when $\varphi $
tends to $ \theta - \pi/2$ from above and
 to $\pi/2 $ from below.
This implies  $f(\varphi) $ has a unique minimizer $\varphi_0$
characterized by the relation $f'(\varphi_0)=0$. Therefore we
conclude that (i) and (ii) are equivalent.

Note also that in this case $O\not \in [A'B']$. This means that $f'(\varphi) = |OA'| - |OB'| =0$ if and only if $A'=B'$. Therefore (i) and (iii) are equivalent.

\end{proof}

\begin{rem} \label{rem1}
{\rm The above statement remains true also in the degenerate case
of a strip bounded by two different parallel lines. Moreover, in this case CPP has a place for every point $O \in [A_0B_0]$ and the line $A_0B_0$ is perpendicular to the lines bounding the strip.
}
\end{rem}

\medskip

Recall that a line $p$ is said to be \emph{supporting} for a set $\mathcal C \subset \mathbb{R}^2$ at some point $A \in \mathcal C$ if $A \in p$
 and  $\mathcal C$ lies entirely in one of the two closed half-planes determined by $p$.
It is well-known that for every point $A$ from the boundary $\partial\, \mathcal C$
of a closed convex set $\mathcal C$ with nonempty interior there exists at least one supporting line at this point.

Let $\mathcal C$ be a compact convex set $\mathcal C$ in $ \mathbb{R}^2$ and $O$ be an interior point of it.
Consider a line $l(\varphi)$ rotating around $O$
 with $\varphi$ as angle of rotation. Put $\tilde{f}(\varphi) = |AB|$ where $[AB]=  l(\varphi)\cap \mathcal{C}$.
 Numerical experiments with GeoGebra show
  that the function  $\tilde{f}(\varphi)$ may have several local minima  and local maxima. For instance, consider the ellipse
  $$\left( \frac{x}{12} \right)^2 + \left( \frac{y}{1.6} \right)^2 = 1$$
  and the point $O(-1,1.4)$ which is interior for the ellipse.
  The graph of the function $\tilde{f}(\varphi)$ is exhibited on the picture below.


  \begin{figure}[h]  \label{Picture1}
\centering
\includegraphics[scale=0.3]{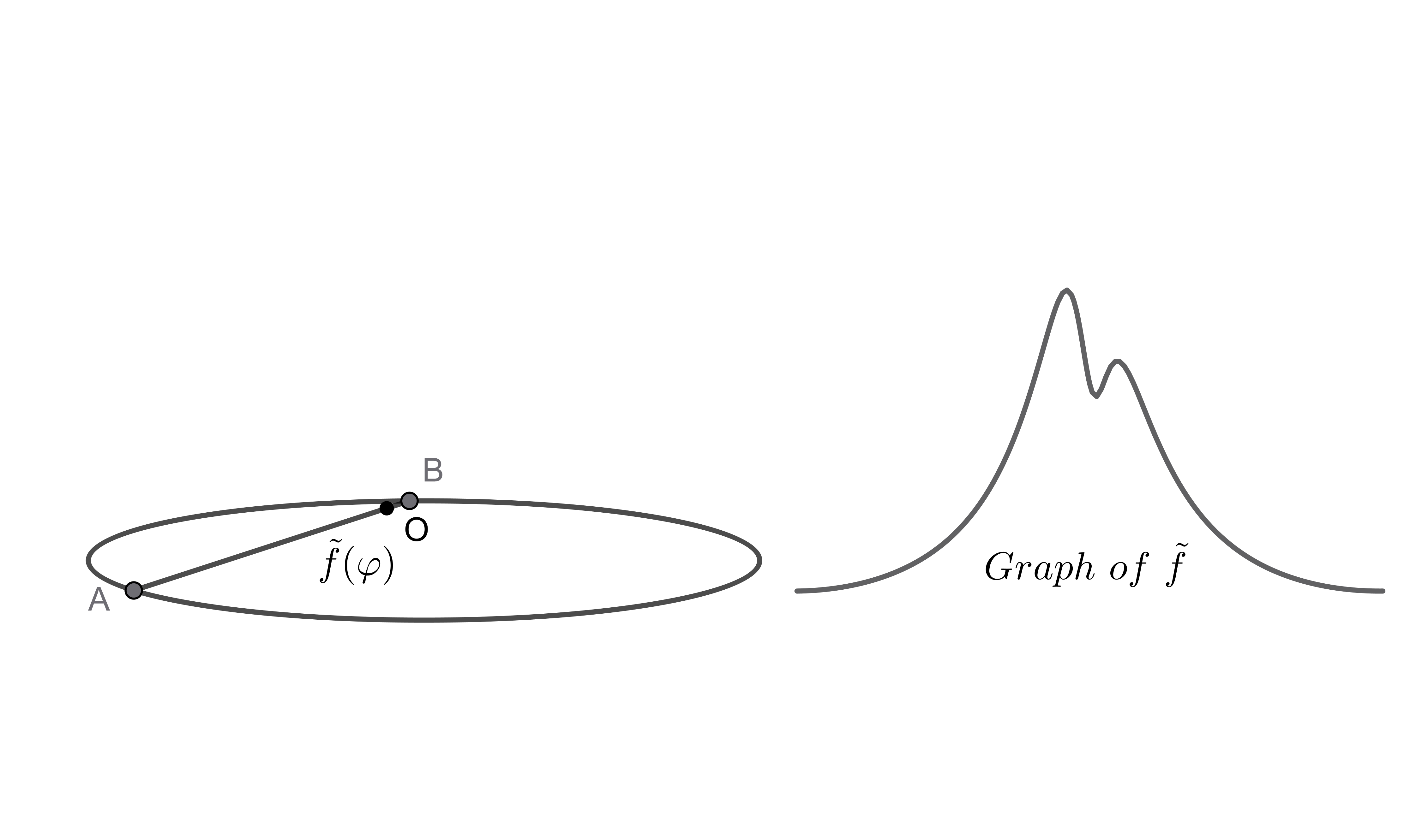}
\caption*{Ellipse with axes  a =12, b = 1.6, and point O(-1,1.4)}
\end{figure}

 The next Proposition gives some information about the local minima of $\tilde{f}(\varphi)$.

\begin{prop} \label{dim2prop}
Let $\mathcal C \subset \mathbb{R}^2$ be a compact convex set,  $O$ be an interior point of it, and
$\tilde{f}(\varphi) := |\tilde{A}(\varphi)\tilde{B}(\varphi)|$ where $[\tilde{A}(\varphi)\tilde{B}(\varphi)]=  l(\varphi)\cap \mathcal{C}$. Suppose $\varphi_*$ is a local minimum for $\tilde{f}(\varphi)$ and let $[A^*B^*]= l(\varphi_*) \cap \mathcal C$  Then:

(i) CPP holds for every pair of supporting for $\mathcal C$ lines
$l_1$ at $A^*$, $l_2$ at $B^*$ and the point $O \in [A^*B^*]$.

(ii) At each of the end points  $A^*$ and $B^*$ there exists only one supporting line for $\mathcal C$.

(iii) $\tilde{f}(\varphi)$ is differentiable at $\varphi_*$ and $\tilde{f}'(\varphi_*) = 0$.

\end{prop}

\begin{proof}
(i) Let ${A}(\varphi)=l(\varphi) \cap l_1$ and ${B}(\varphi)=
l(\varphi) \cap l_2$. For every $\varphi$ such that $l(\varphi)$
intersects both $l_1$ and $l_2$  we have
$[\tilde{A}(\varphi)\tilde{B}(\varphi)]\subseteq
[A(\varphi)B(\varphi)] \cap \mathcal C$. In particular,
${A}(\varphi_*) = A^*$ and ${B}(\varphi_*) = B^*$. Obviously, $
|A(\varphi)B(\varphi)| \geq |\tilde{A}(\varphi)
\tilde{B}(\varphi)|  $. If $\varphi$ is close enough to
$\varphi_*$ we have

$$ |A(\varphi)B(\varphi)| \geq |\tilde{A}(\varphi) \tilde{B}(\varphi)| \geq |A^*B^*|= |A(\varphi_*)B(\varphi_*)|.$$

\noindent Hence $\varphi_*$ is a local minimum for the function $
f(\varphi)=|{A}(\varphi) {B}(\varphi)|$ which, as we have seen in
Proposition \ref{mainprop}, is strictly convex. Therefore,
$\varphi_*$ is a global minimum for $f(\varphi)$ and the CPP we
want to prove follows from property (iii) of Proposition
\ref{mainprop}.

(ii) Let $l_1$ be some supporting for $\mathcal C$ line at $A^*$.
Suppose $\mathcal C$ has two supporting lines $l'_2$ and $l''_2$
at $B^*$. Denote by ${A^*}'$ the intersection point of the line
perpendicular to $l_1$ at $A^*$ and the line perpendicular to the
line $A^*B^*$ at $O$. It follows from (i) that ${A^*}'B^*$ is
perpendicular to both $l'_2$ and $l''_2$. This is possible only if
$l'_2$ and $l''_2$ coincide. Thus we have shown that there is only
one supporting line for  the set $\mathcal C$ at the point $B^*$.
Similarly, one shows that the supporting line for $\mathcal C$ at
$A^*$ is unique as well.

\smallskip

Property (iii) follows from Lemma \ref{linearization}, which will be proved later.
\end{proof}

\medskip

The proof of Proposition \ref{dim2prop} is also valid, with minor changes, for unbounded closed convex sets $\mathcal C$ with nonempty interiors. For such a set, the function $\tilde{f}$ may be equal to $+ \infty$ at some $\varphi$.

\begin{cor}
Let  $\mathcal C$ be a closed convex set and $O$ be an interior point of it. If some $\varphi_*$ belonging to the set $\{\varphi : \tilde{f}(\varphi) < + \infty\}$ is local minimizer for $\tilde{f}(\varphi)$,
then (i), (ii) and (iii) from Proposition \ref{dim2prop} are valid.
 \end{cor}

As mentioned in the introduction, Proposition \ref{dim2prop} is not valid for local maximizers.

\begin{prop} \label{polygon}
Let $\mathcal C$ be a convex polygon in $\mathbb R^2$, $O$ an interior point of it. Let
$\varphi^*$  be a local maximizer for the function
$\tilde{f}(\varphi)$ with $l(\varphi^*) \cap \mathcal C =
[A^*B^*]$.  Then at least one of the end-points $A^*$ and $B^*$ is
a vertex of $\mathcal{C}$.

\end{prop}

\begin{proof}
Suppose both $A^*$ and $B^*$ are not vertices of $\mathcal C$.
Then they are interior points of some sides $[A'A'']$ and $[B'B'']$
of $\mathcal C$. If the lines $l_1=A'A''$ and $l_2 =B'B''$
intersect at some point $E$, then $O$ is an interior point for the
angle $\angle Ol_1l_2$  and the strictly convex function
$f(\varphi)$  from Proposition \ref{Rem1} can not have local
maximizers, a contradiction. The same argument holds also if the
lines $l_1$ and $l_2$ are parallel (see Remark \ref{rem1}).

\end{proof}

 Our next immediate goal is to extend the validity of CPP by replacing the lines $l_1, l_2$ with arbitrary smooth curves. This idea is rather old. Botema \cite{Bot} notes that it was mentioned in M. Lhuilier's book \cite{ML} from 1795 (but not mentioned in the later article by Lhuilier \cite{ML1} from 1811). H.Eves says in \cite{EVES} ``This generalization seems to be due to Isaac Newton''. Both O. Botema and H. Eves express concerns that some of the proofs of the results related to the replacement of straight lines by curves do not correspond to the contemporary requirements for rigor.  To make this article self-contained and to remove the mentioned concerns we provide proofs of the facts needed for our considetations.


\begin{defi}
Let $p$ be a straight line and $\gamma$ a curve in $\mathbb{R}^2
$. Let the point $P_0$ belongs to both $p$ and $\gamma$. We say
that $p$ is a tangent to $\gamma \subset \mathbb{R}^2 $ at the
point $P_ 0$ if there exists an open circle $U$ in $\mathbb{R}^2$
centered at $P_0$ such that for a rectangular coordinate system
with  $p$ as abscissa axis, the set $\gamma \cap U$ is the graph
of a continuous real-valued function which is differentiable at
$P_0$ and its derivative at this point vanishes.

We say that the curve $\gamma$ is smooth if it has a tangent at every point.
\end{defi}

\smallskip

The curves $\gamma$ we mostly use in this article are parts of the boundaries  $\partial \,  \mathcal{C}$ of closed convex sets  $\mathcal{C} \subset \mathbb{R}^2$ with nonempty interior.
Given a point $O$ (inside or outside the set
$\mathcal C$),
a point $\tilde{A}(\varphi)$ of such a
curve is determined by its distance
$|O\tilde{A}(\varphi)|$ from the point $O$
and by the
angle $\varphi$ the ray
$\overrightarrow{O\tilde{A}(\varphi)}$ concludes with a suitably chosen direction.
Such curves are, therefore, represented by
a function $\tilde{\rho}(\varphi):= |O\tilde{A}(\varphi)|$
 where $\varphi$ runs over some interval.

\smallskip

Suppose a straight line $p$ is tangent at
some point $P_0$  to a curve $\gamma$ (given by some function
$\tilde{\rho}(\varphi):= |O\tilde{A}(\varphi)|$)
and $O$ is a point outside $p \,\cup \gamma$ .
Denote, as above, by $O'$ the orthogonal projection of $O$ to $p$.
Without loss of generality we may assume that the ray $\overrightarrow{OO'}$ is the ``suitably chosen direction'' and
$\varphi = \angle \tilde{A}(\varphi)OO'$
(Figure 4). Put $\varphi_0:= \angle P_0OO'$.

When the angle $\varphi$ is close enough to $\varphi_0$ the ray
$\overrightarrow{O\tilde{A}(\varphi)}$ intersects the line
$p$ at some point $A(\varphi)$.
Put $\rho(\varphi):=|OA(\varphi)|$. Clearly, $\tilde{\rho}(\varphi_0)= \rho(\varphi_0) = |OP_0|$.

\begin{figure}[h]  \label{f4}
\centering
\includegraphics[scale=0.6]{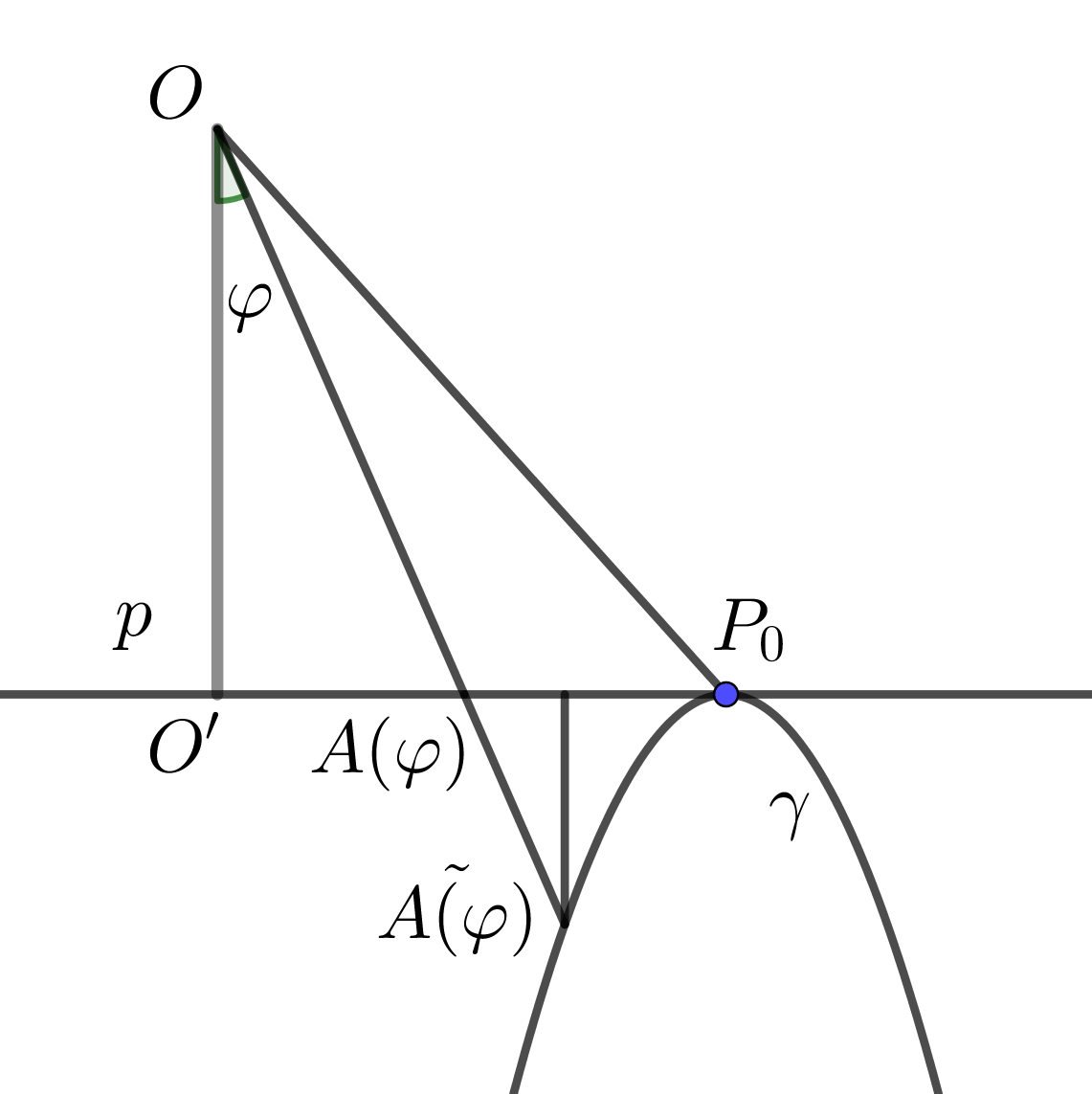}
\caption*{Figure 4}
\end{figure}

\begin{lemma} \label{linearization}
The functions $\tilde{\rho}(\varphi)$ and $\rho(\varphi)$ are
differentiable at $\varphi_0$ and
$\tilde{\rho}'(\varphi_0)=\rho'(\varphi_0)$.
\end{lemma}
\begin{proof}
Consider the rectangular coordinate system with origin $O'$ and
abscissa axis $p$.  Since $p$ is tangent to $\gamma$ at $P_0$, we
may assume that in small open circle centered at $P_0$ the curve
$\gamma$ is the graph of a continuous  real-valued function
$h(t)$, defined in an open interval $\Delta \subset p$ containing
$P_0=(p_0,0)$ and such that $h(p_0)=h'(p_0)=0$.

Then $\tilde{A}(\varphi)=(t,h(t)), t\in \Delta$ and we may assume without restricting the generality that $|OO'|=1$.
Then
$\displaystyle\tilde{\rho}(\varphi)=(1-h(t))\rho(\varphi)$ and we
get

$$\frac{\tilde{\rho}(\varphi)-\tilde{\rho}(\varphi_0)}{\varphi-\varphi_0}=
(1-h(t))\cdot\frac{\rho(\varphi)-\rho(\varphi_0)}{\varphi-\varphi_0}
-\frac{h(t)-h(p_0)}{\varphi-\varphi_0}\cdot\rho(\varphi_0) .$$
Note that $\lim_{\varphi \rightarrow \varphi_0}h(t)=\lim_{t
\rightarrow p_0}h(t)= h(p_0)=0$ and we have to prove that
$$\lim_{\varphi \rightarrow
\varphi_0}\frac{h(t)-h(p_0)}{\varphi-\varphi_0}=0.$$ To do this
observe that $\displaystyle \tan(\varphi)=\frac{t}{1-h(t)}$ which
implies the identity

$$\frac{h(t)-h(p_0)}{\varphi-\varphi_0}=\frac{\tan(\varphi)-\tan(\varphi_0)}{\varphi-\varphi_0}
\cdot\displaystyle
\frac{\displaystyle\frac{h(t)-h(p_0)}{t-p_0}}{1+\tan(\varphi)\cdot\displaystyle\frac{h(t)-h(p_0)}{t-p_0}}.$$

Having in mind that $$\lim_{\varphi
\rightarrow\varphi_0}\frac{\tan(\varphi)-\tan(\varphi_0)}{\varphi-\varphi_0}=\cos^{-2}(\varphi_0)
$$ and   $$\lim_{t \rightarrow
p_0}\frac{h(t)-h(p_0)}{t-p_0}=h'(p_0)=0$$ we obtain from the above
identity that
$$\lim_{\varphi \rightarrow
\varphi_0}\frac{h(t)-h(p_0)}{\varphi-\varphi_0}=0$$ and we are
done.
\end{proof}

\medskip

It is known (cf.~\cite[Theorem 25.1]{Rock}) that a convex set has a tangent $p$ at some $P_0 \in \partial\, \mathcal{C}$ if and only if $p$ is the only supporting line for $\mathcal{C} $ at $P_0$.

We are now ready to prove item (iii) of Proposition
\ref{dim2prop}:

\begin{proof} (of item (iii) of Proposition \ref{dim2prop}) The unique line $t_1$  supporting $\mathcal C$ at $A^*$ is tangent to $\partial \, C$ at $A^*$.
Lemma \ref{linearization} implies that  the function $|O \tilde{A}(\varphi)|$ is differentiable at $\varphi_*.$
Using similar arguments, we see that the function $|O \tilde{B}(\varphi)|$ is also differentiable at $\varphi_*$. This proves the differentiability of the function
$$\tilde{f}(\varphi)=|\tilde{A}(\varphi)\tilde{B}(\varphi)| = |O \tilde{A}(\varphi)| +|O \tilde{B}(\varphi)|$$

As $\varphi_*$ is a local minimizer of $\tilde{f}(\varphi)$ we get
that $\tilde{f}'(\varphi_*) = 0.$

\end{proof}

We call a convex set $\mathcal C \subset \mathbb R^2$ {\it smooth} if there exists a
tangent at every point of its
boudary $\partial \, C$.

 \begin{prop} \label{max}
 Let $\mathcal D$ be a smooth compact convex subset of $\mathbb{R}^2$ and $O$ an interior point of it.
 As above, consider all straight lines $l(\varphi)$ through $O$ and the function
 $\tilde{f}(\varphi)=|\tilde{A}(\varphi)\tilde{B}(\varphi)|$ where $[\tilde{A}(\varphi)\tilde{B}(\varphi)]= l(\varphi) \cap \mathcal D$.
 Denote by $t_1(\varphi)$ and $t_2(\varphi)$ the tangents to $\partial \,\mathcal D$ at $\tilde{A}(\varphi)$ and $\tilde{B}(\varphi)$, respectively.Then:

 (i) The function $\tilde{f}(\varphi)$ is everywhere differentiable.

 (ii) $\tilde{f}'(\varphi)=0$ if and only if
  CPP holds for the lines $t_1(\varphi)$, $t_2(\varphi)$, $l(\varphi)$ and the points $\tilde{A}(\varphi)$, $\tilde{B}(\varphi)$, $O$.

\end{prop}

\begin{proof}
Take an arbitrary $\varphi_0$  and put $A(\varphi):= l(\varphi)
\cap t_1(\varphi_0)$, $B(\varphi):= l(\varphi) \cap
t_2(\varphi_0)$. Consider the function
$f(\varphi):=|A(\varphi)B(\varphi)|=|OA(\varphi)| +|OB(\varphi)|$.
It is well defined and differentiable whenever $l(\varphi)$ is
parallel neither to $t_1(\varphi_0)$ nor to $t_2(\varphi_0)$. In
particular, this is so for $\varphi = \varphi_0$. By Lemma
\ref{linearization}, the function
 $\tilde{f}(\varphi) =|\tilde{A}(\varphi)\tilde{B}(\varphi)| = |\tilde{A}(\varphi)
 O| + |O \tilde{B}(\varphi)|$
 is differentiable at $\varphi_0$ (this proves the validity of (i) ), and its derivative coincides with the derivative of the function
 $f(\varphi)=|A(\varphi)B(\varphi)|$ at $\varphi_0$. Hence, the condition $\tilde{f}'(\varphi_0)=0$ is equivalent to the requirement ${f}'(\varphi_0)=0.$
  In view of Proposition \ref{mainprop}, this is precisely the case when CPP holds for the lines $t_1(\varphi_0)$, $t_2(\varphi_0)$, $l(\varphi_0)$
   and the points $\tilde{A}(\varphi_0) = A(\varphi_0)$, $\tilde{B}(\varphi_0)= B(\varphi_0), O$.
\end{proof}

\medskip

\begin{cor}
Let $\mathcal D$ and $\tilde{f}(\varphi)$ be as in Proposition
\ref{max}. Then CPP holds for every local maximizer $\varphi^*$
and every local minimizer $\varphi_*$ of this function.
\end{cor}

\begin{rem} {\rm  (i)  Pay attention to the seemingly strange phenomenon connected with the proof of Proposition \ref{max}: If $\varphi_0$ is a
local maximizer for the function $\tilde{f}(\varphi)$, it is also
a global minimizer for the closely related function $f(\varphi)$
which, in a sense, is a kind of ``twofold  linearization'' of
$\tilde{f}(\varphi)$ at $\varphi_0$
($|OA(\varphi)|$ is a linearization of $|O\tilde{A}(\varphi)|$  and
$|OB(\varphi)|$ is a linearization of
$|O\tilde{B}(\varphi)|$).

(ii) Items (i) and (ii) of Proposition \ref{max} remain valid also for unbounded smooth closed convex sets $\mathcal D$ with nonempty interiors
whenever  $\varphi$ satisfies  $\tilde{f}(\varphi) < +\infty.$}
\end{rem}

At the end of this subsection we prove a lemma that will be used in the next section.

\begin{lemma} \label{observe}
Let $\mathcal D$ and $\tilde{f}(\varphi)=|\tilde{A}(\varphi)\tilde{B}(\varphi)|$
be as in Proposition \ref{max}. Let $\varphi^*$ be a local maximizer for $\tilde{f}(\varphi)$
and let the tangent $t_1(\varphi^*)$ to $\partial \, \mathcal D$ at
$\tilde{A}(\varphi^*)$ be parallel to the tangent
$t_2(\varphi^*)$ to $\partial \, \mathcal D$ at
$\tilde{B}(\varphi^*)$. Then the line $\tilde{A}(\varphi^*)\tilde{B}(\varphi^*)$ is perpendicular to the tangents.
\end{lemma}

\begin{proof}
As in the proof of Proposition \ref{max}, define
$A(\varphi):= l(\varphi)
\cap t_1(\varphi^*)$ and $B(\varphi):= l(\varphi) \cap t_2(\varphi^*)$. We know that the function
$f(\varphi):=|A(\varphi)B(\varphi)|=|OA(\varphi)| +|OB(\varphi)|$ is strictly convex and differentiable. It attains its global minimum
at the point where its derivative is zero and this happens precisely when $l(\varphi)$ is perpendicular to the lines $t_1(\varphi^*)$ and
$t_1(\varphi^*)$. Note that the function
 $\tilde{f}(\varphi) =|\tilde{A}(\varphi)\tilde{B}(\varphi)| = |\tilde{A}(\varphi)
 O| + |O \tilde{B}(\varphi)|$
 is differentiable at $\varphi^*$ and its derivative at this point coincides with the derivative of the function
 $f(\varphi)=|A(\varphi)B(\varphi)|$ at the same point. As $\varphi^*$ is a local maximizer for $\tilde{f}(\varphi)$ we get
$f'(\varphi^*) = \tilde{f}'(\varphi^*) = 0$.
 \end{proof}

 As seen from its proof, this lemma is also valide for unbounded smooth closed convex sets $\mathcal C$ with nonempty interior.

\bigskip

\subsection{Point $O$ is outside the set $\mathcal C$}

We consider first the case when $\mathcal C$ is an angle $\angle
Ol_1l_2$ with $0<\theta :=\angle Ol_1l_2 <\pi/2$,
and
$O$ is outside the angle and its opposite angle. This case is
presented in Figure 5 where $O'$ and $O''$ are the orthogonal
projections of $O$ onto the lines containing the arms $l_1$ and
$l_2$, respectively. A line $l(\varphi)$ (the angle $\varphi$ will
be specified next) passes through $O$ and intersects the arms of
the angle at the points $A=A(\varphi)$ and $B=B(\varphi)$ ($B \in
[OA]$). We put  $\varphi:= \angle AOO'$ if $O'$ belongs to the ray
$\overrightarrow{AE}$ and $\varphi:= -\angle AOO'$ otherwise.

\begin{figure}[h]  \label{f5}
\centering
\includegraphics[scale=0.3]{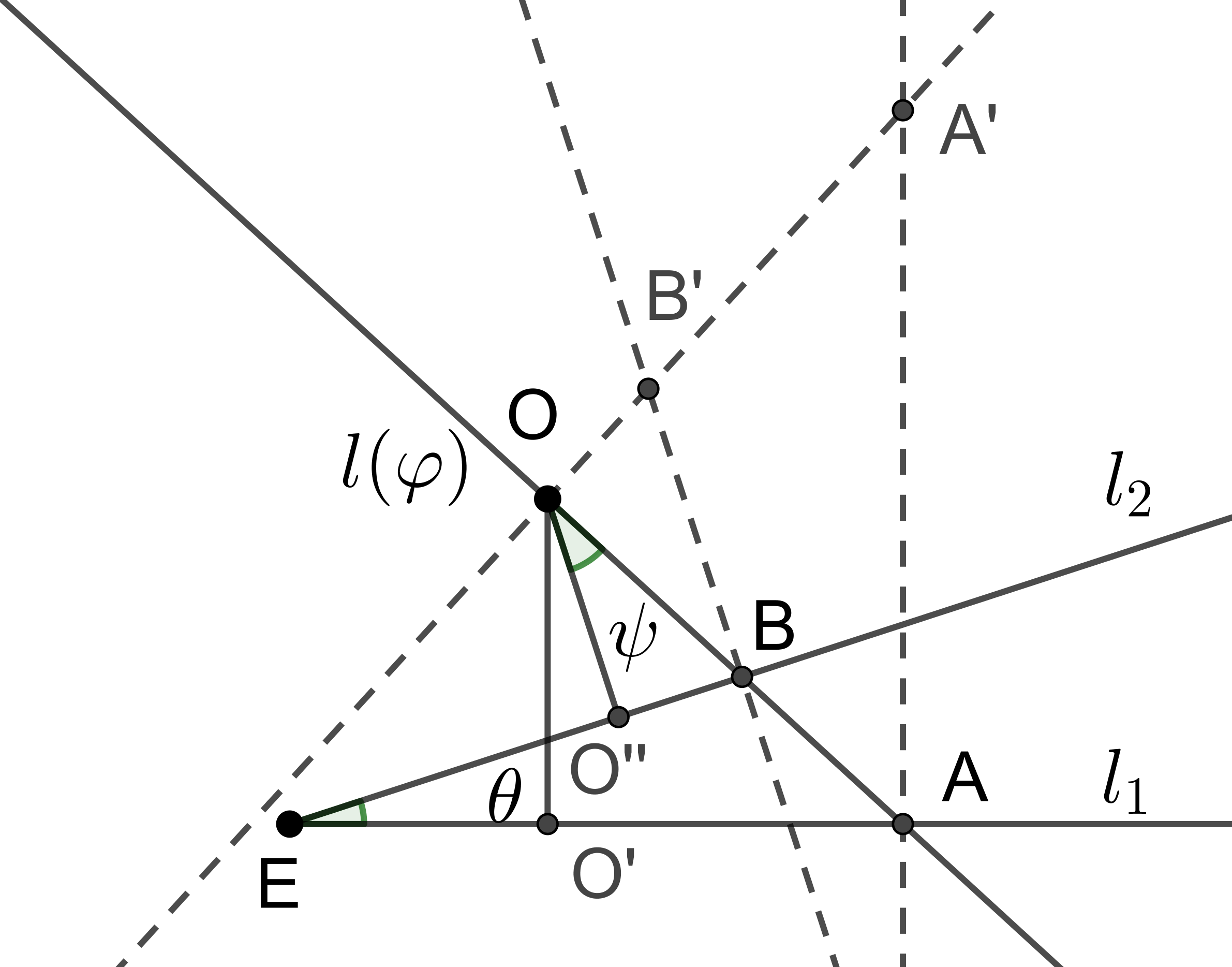}
\caption*{Figure 5}
\end{figure}

We are interested in the behavior of the function
$$g(\varphi):= |OA(\varphi)| - |OB(\varphi)| = |A(\varphi)B(\varphi)|$$
which is well-defined and differentiable in its domain (the open interval \\
$(-\pi/2 + \angle OEO', \pi /2)$).

As mentioned in the Introduction,  Neovius \cite{EN} has shown
that depending on the angle measure of the angle $\angle OEO'$ the
function $g'(\varphi)$ may have $0, 1,$ or $2$ critical points  in
the interval $(-\pi/2 + \angle OEO', 0)$. In the third case, the two critical points correspond to a local maximum and a local minimum of $g(\varphi)$.

\begin{prop} \label{CPP O outside}
(CPP for angles.) Let $\varphi$ be from the domain of the function
$g$. Then $g'(\varphi)=0$ if and only if CPP holds for the line
$l(\varphi)$, the arms $l_1$, and $ l_2$ and the points $O$,
$A(\varphi)$, and  $B(\varphi)$ , respectively.

In particular, CPP holds for some $\varphi$ whenever the function
$g$ has a local minimum or a local maximum at $\varphi$.
\end{prop}

\begin{proof}
The dashed lines in Figure 5 are orthogonal to: $l_1$ at $A=A(\varphi)$, $l_2$ at $B=B(\varphi)$, and $l(\varphi)$ at $O$. Point $A'=A'(\varphi)$ is the intersection of the dashed line orthogonal to $l(\varphi)$ and the dashed line orthogonal to $l_1$. Similarly, point $B'= B'(\varphi)$ is the intersection point of the dashed lines orthogonal to $l(\varphi)$ and $l_2$.

Consider the angle $\psi = \varphi - \theta$ concluded between
$l(\varphi)$ and the line $OO''$. Clearly, $|OB|=
|OO''|\cos^{-1}\psi$. Hence, by Lemma \ref{L1},
$$g'(\varphi)= \sign(\varphi)|OA'| - \sign(\varphi - \theta)|OB'|.$$
Note that $\theta = \angle O''OO'$. If $0\leq \varphi \leq \theta$, then $g'(\varphi) =|OA'| + |OB'|> 0$. If $\varphi > \theta$ we have
$$g'(\varphi) =|OA'| - |OB'| = |OA|\tan \varphi - |OB|\tan(\varphi - \theta)>|OB|(\tan \varphi - \tan(\varphi - \theta)) >0.$$

This shows that $g(\varphi)$ increases strictly when $\varphi >0$. For $\varphi$ from the interval $(-\pi/2 + \angle OEO, 0)$ we have
$$g'(\varphi) = - |OA'| +|OB'|.$$
Note also that in this case the segment $[A'B']$ does not contain the point $O$. This means $g'(\varphi)=0$ if and only if $A'=B'$ which is the CPP.
\end{proof}

\begin{rem} \label{ostar}
{\rm
 It is easy to realize that, when $\pi/2 \leq \theta < \pi$, the behavior of the function $g(\varphi)$ is not interesting. It is strictly increasing in its domain; therefore, $g'(\varphi)$ does not have zeros. }
\end{rem}

Let now $\mathcal C$ be a closed convex subset (not necessarily bounded) of $\mathbb{R}^2$ with a
 nonempty interior and $O$ be a point outside  $\mathcal{C}$. Consider all
  straight lines $l(\varphi)$ through $O$ and put  $D :=\{\varphi: l(\varphi) \cap \mathcal C  = [\tilde{A}(\varphi)\tilde{B}(\varphi)] \not = \emptyset$ and $\tilde{B}(\varphi) \in [O\tilde{A}(\varphi)] \}$.
  The set $D$ is the domain of the function
 $\tilde{g}(\varphi):= |O \tilde{A}(\varphi)| - |O\tilde{B}(\varphi)| = |\tilde{A}(\varphi) \tilde{B}(\varphi)|.$
The next two propositions (and their proofs) are similar to Proposition \ref{dim2prop} and Proposition \ref{max}.

 \begin{prop} \label{tilde O outside}
 If the function $\tilde{g}(\varphi)$ has a local minimum at some $\varphi_*$ from
the interior of its domain $D$ and $l(\varphi_*) \cap \mathcal{C} = [A^*B^*], B^* \in [OA^*]$, then

(i) $\mathcal{C}$ has
unique supporting line $t_1$ at $A^*$ and unique supporting line $t_2$ at $ B^*$. I.e., $t_1$ is a tangent to $\partial\, C$ at $A^*$ and $t_2$ is a tangent to $\partial \, C$ at $B^*$.

(ii) CPP holds for the lines $t_1, t_2, l(\varphi_*)$  and the
points $A^*$, $B^*$ and $O$.
\end{prop}

\begin{proof}
Let  $l_1$ and $l_2$ be any supporting lines for $\mathcal C$ at
the points $A^*$ and $B^*$ , respectively. For $\varphi \in D$ put
$A(\varphi):=l(\varphi) \cap l_1$
 and
$B(\varphi):= l(\varphi) \cap l_2$. In particular, ${A}(\varphi_*) = A^*$ and ${B}(\varphi_*) = B^*$. Also,
we have $[\tilde{A}(\varphi)\tilde{B}(\varphi)]= [A(\varphi)B(\varphi)] \cap \mathcal C$ and, therefore, $ |A(\varphi)B(\varphi)| \geq |\tilde{A}(\varphi) \tilde{B}(\varphi)|  $.
If $\varphi$ is close enough to $\varphi_*$ we have

$$ |A(\varphi)B(\varphi)| \geq |\tilde{A}(\varphi) \tilde{B}(\varphi)| \geq |A^*B^*|= |A(\varphi_*)B(\varphi_*)|.$$

\noindent I.e., $\varphi_*$ is a local minimum for the function $
g(\varphi)=|{A}(\varphi) {B}(\varphi)|$ from Proposition \ref{CPP
O outside}. It follows that CPP holds for the lines $l_1, l_2,
l(\varphi_*)$  and the points $A^*$, $B^*$ and $O$.  This implies,
precisely as in the proof of Proposition \ref{dim2prop}, that
$\mathcal C$ has unique supporting lines at $A^*$ and $B^*$.
\end{proof}

Let $\mathcal D$ be a smooth closed convex subset of
$\mathbb{R}^2$ with nonempty interior and $O$ a point outside it.
Consider as above  all straight lines $l(\varphi)$ through $O$ and
put  $D :=\{\varphi: l(\varphi) \cap \mathcal D  =
[\tilde{A}(\varphi)\tilde{B}(\varphi)] \not = \emptyset$ and
$\tilde{B}(\varphi) \in [O\tilde{A}(\varphi)] \}$.
  The set $D$ is the domain of the function
 $\tilde{g}(\varphi):= |O \tilde{A}(\varphi)| - |O\tilde{B}(\varphi)| = |\tilde{A}(\varphi) \tilde{B}(\varphi)|.$
 For every $\varphi \in D$  denote by $t_1(\varphi)$ and $t_2(\varphi)$
 the the tangents to $\partial \,\mathcal D$ at $\tilde{A}(\varphi)$ and $\tilde{B}(\varphi)$, respectively

\begin{prop} \label{max O outside}
The following statements have place for every $\varphi$ from the interior of $D$:

 (i) The function $\tilde{g}(\varphi)$ is differentiable at $\varphi$.

 (ii) $\tilde{g}'(\varphi)=0$ if and only if
  CPP holds for the lines $t_1(\varphi)$, $t_2(\varphi)$, $l(\varphi)$, and the points $\tilde{A}(\varphi), \tilde{B}(\varphi), O$, respectively.

  In particular, CPP holds for every local minimizer $\varphi_*$ and every local maximizer $\varphi^*$ of the function $\tilde{g}$
(whenever $\varphi_*$ and $\varphi^*$ are from the interior of the domain $D$ of $\tilde{g}$).

\end{prop}

\begin{proof}
Take an arbitrary $\varphi_0$  from the interior of $D$ and put
$A(\varphi):= l(\varphi) \cap t_1(\varphi_0)$, $B(\varphi):=
l(\varphi) \cap t_2(\varphi_0)$. Consider the function
$$g(\varphi):=|A(\varphi)B(\varphi)|=|OA(\varphi)| - |OB(\varphi)|.$$
It is well defined and differentiable at $\varphi = \varphi_0$. By Lemma \ref{linearization}, the function
 $\tilde{g}(\varphi) =|\tilde{A}(\varphi)\tilde{B}(\varphi)| = |\tilde{A}(\varphi)
 O| - |O \tilde{B}(\varphi)|$
 is differentiable at $\varphi_0$ (this proves the validity of (i)), and its derivative coincides with the derivative of the function
 $g(\varphi)=|A(\varphi)B(\varphi)|$ at $\varphi_0$. Hence, the condition $\tilde{g}'(\varphi_0)=0$ is equivalent to the requirement ${g}'(\varphi_0)=0.$   In vew of Proposition \ref{CPP O outside}, this is precisely the case when CPP holds for the lines $t_1(\varphi_0)$, $t_2(\varphi_0)$, $l(\varphi_0)$, and the points $\tilde{A}(\varphi_0) = A(\varphi_0)$, $\tilde{B}(\varphi_0)= B(\varphi_0)$, and $O$.
\end{proof}
An analog of Lemma \ref{observe} is valid in the case when $O$ is outside the set $\mathcal D$.

\begin{lemma} \label{observe1}
Let $\mathcal D$ and $\tilde{g}(\varphi)=|\tilde{A}(\varphi)\tilde{B}(\varphi)|$
be as in Proposition \ref{max O outside}.
Let $\varphi^*$ from the interior of the domain of $\tilde{g}$ be a local maximizer
for $\tilde{g}(\varphi)$ Suppose
the tangent $t_1(\varphi^*)$ to $\partial \, \mathcal D$ at
$\tilde{A}(\varphi^*)$ is parallel to the tangent
$t_2(\varphi^*)$ to $\partial \, \mathcal D$ at
$\tilde{B}(\varphi^*)$. Then the line $\tilde{A}(\varphi^*)\tilde{B}(\varphi^*)$ is perpendicular to the tangents.
\end{lemma}

The proof of this statement is omitted because it is quite similar to the proof of Lemma \ref{observe}.

\begin{rem}
{\rm The assumption about the convexity of
the set $\mathcal D$ in Proposition \ref{max O outside} and  in
Proposition \ref{max}  is not essential for the validity of the statements.
It suffices to consider two smooth curves
$\gamma_1$ and $\gamma_2$ and a point $O$ as in the next picture.

\begin{figure}[h]  \label{Pic 3}
\centering
\includegraphics[scale=0.4]{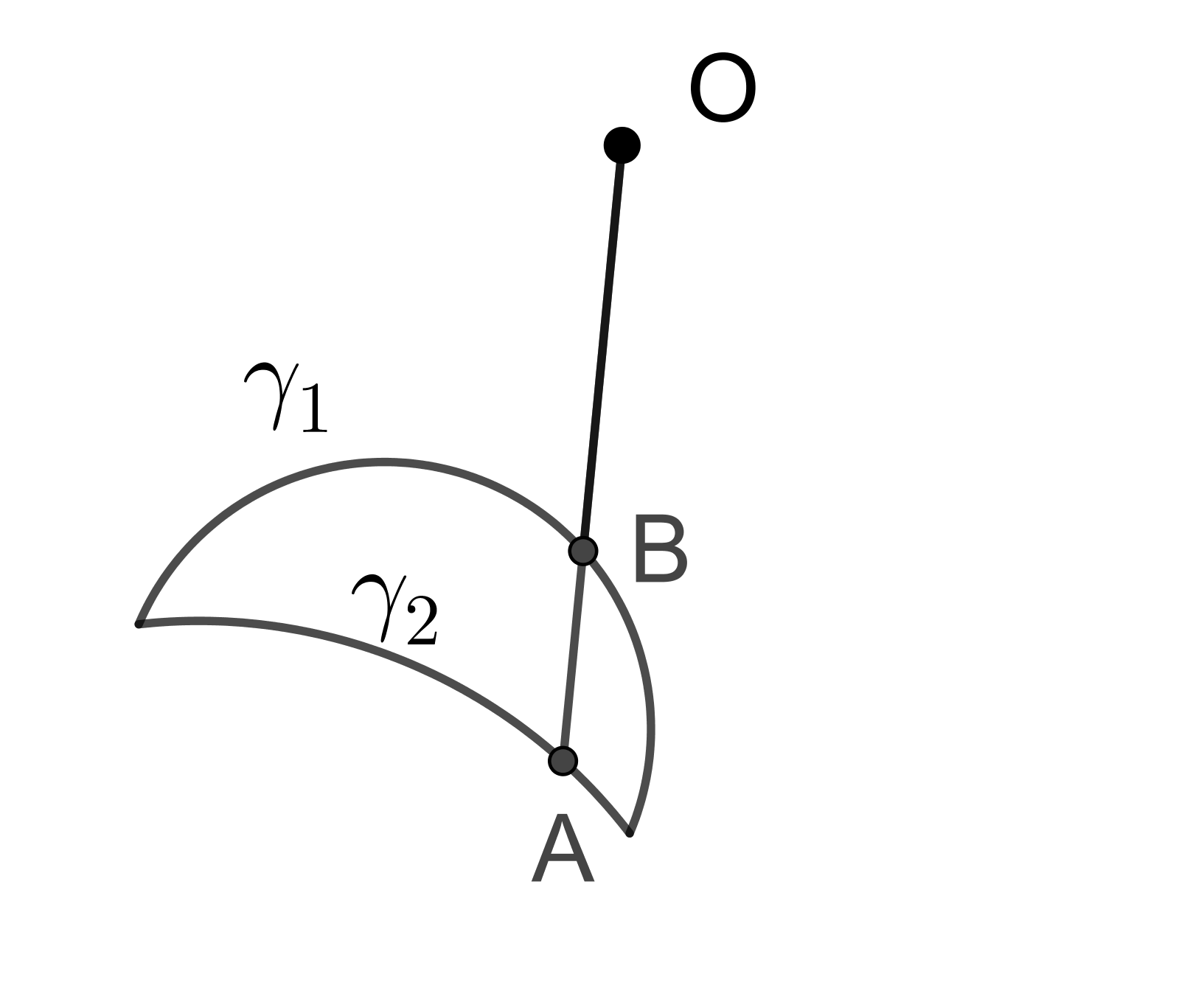}
\caption*{Arbitrary smooth curves $\gamma_1$ and $\gamma_2$}
\end{figure}

\vspace{6cm}

Then, CPP has a place for every local maximizer  $\varphi^*$
or a local minimizer $\varphi_*$ of the function
$\tilde{g}(\varphi) = |O \tilde{A}(\varphi)| - |O \tilde{B}(\varphi)|$ whenever $\varphi^*$ and $\varphi_*$
belong to the interior of the domain of $\tilde{g}$. Similar is the situation with the function
$\tilde{f}(\varphi) = |O\tilde{A}(\varphi)|
 + |O \tilde{B}(\varphi)|$.}
\end{rem}

\section{The case when the set $\mathcal{C}$ is in $\mathbb{R}^n, n >2 $ .}

In this section, we show how the 2-dimensional statements established in the previous section imply their higher dimensional generalizations, including the results announced in the paper's abstract.

Let $\mathcal C \subset \mathbb{R}^n,\, n > 2,$ be a closed convex set and $O$ be a point from its interior.
Consider all segments $[AB]$ containing $O$
 with end-points $A$ and $B$ from the boundary $\partial \, \mathcal C$, and define for such segments the  function  $F([AB])=|AB|$.

\begin{defi} \label{defi-min-max}
A segment $[A^*B^*]$ containing $O$, with end points on
$\partial \, \mathcal C$, is said to be a \emph{local minimizer
(maximizer)} for $F$ if there are open sets $U \ni A^*$ and $V \ni
B^*$ such that $F([AB]) = |AB| \geq
 |A^*B^*|$ ($F([AB]) = |AB| \leq
 |A^*B^*|)$ whenever $O \in [AB]$, $A \in \partial \, \mathcal C \cap U$ and $B \in \partial \, \mathcal C \cap V$.
\end{defi}

This definition agrees with the definitions of a local
minimizer (maximizer) for the function $f(\varphi)$ used in Section 2.
\smallskip

 Recall that a hyperplane $\Pi$ is \emph{supporting a set $\mathcal C$ at a point  $A \in \mathcal C$} if $A \in\Pi$ and $\mathcal C$
 is contained in one of the two closed half-spaces determined by $\Pi$.
 It is well known that for every point $A$ from the boundary of a closed convex set $\mathcal C \subset \mathbb{R}^n$,
 there exists at least one hyperplane that supports $\mathcal C$ at $A$. We call a convex set $\mathcal C$ with non-empty
 interior {\it smooth} if there is only one supporting hyperplane at every point of its boundary.

\begin{thm} \label{thmmin}
Let $O$ be an interior point of a closed convex set $\mathcal C
\subset \mathbb{R}^n,\, n \geq 2$, and let the segment $[A^*B^*]$
be a local minimizer for $F$. Then:

(i) There exists only one hyperplane $\Pi_1$   supporting $\mathcal C$   at  $A^*$ and only one hyperplane $\Pi_2$ supporting $\mathcal C$  at $B^*$;

(ii) The normal line to $\Pi_1$ at $A^*$ and the normal line to
$\Pi_2$ at $B^*$ intersect at a point belonging to the hyperplane
through $O$ which is orthogonal to the line $A^*B^*$.
\end{thm}

\begin{proof}
The particular case of this theorem when $n=2$ follows from
Proposition \ref{dim2prop}. Therefore, we will assume that $n
> 2$. Our first step is to show that  (ii) holds for arbitrary hyperplanes
$\Pi_1$ and $\Pi_2$ supporting $\mathcal C$ at $A^*$ and $B^*$ ,
respectively. Suppose first that $\Pi_1$ and $\Pi_2$ are not
parallel. Then there exist a linear subspace $L \subset
\mathbb{R}^n$ of co-dimension two and a point $T \in \Pi_1 \cap
\Pi_2 $ such that $\Pi_1 \cap \Pi_2 = \{T + h: h \in L\} = T + L$.
Take an arbitrary non-zero vector $h \in L $ and consider the
two-dimensional plane $\Pi(h)$ containing the lines $l_1(h)=\{A^*
+ th: t \in R\}$ and $l_2(h)=\{B^* + th: t \in R\}$. It is clear
that the parallel lines  $l_1(h)$ and $l_2(h)$ are supporting for
the set $\mathcal C \cap \Pi(h)$ in the plane $\Pi(h)$. Moreover,
$[A^*B^*]$ is a local minimizer for the the function $\tilde{f}$ corresponding to the set $\Pi(h) \cap \mathcal
C $. Then, as we have seen in the previous section, the segment $[A^*B^*]$ is a local minimizer for the function $f$ corresponding to the strip bounded by the lines $l_1(h)$ and $l_2(h)$. This means that $A^*B^*$ is perpendicular to  $l_1(h)$ and $l_2(h)$.
Since $h$ was an arbitrary vector from $L$ we
conclude that the line $A^*B^*$ is orthogonal to $L$ (and to
$T+L$).
 Hence there exists a two dimensional plane $\Pi$ which is orthogonal to $T+ L$ and contains the line $A^*B^*$. Put $l_1 = \Pi_1 \cap \Pi$ and $l_2 =\Pi_2 \cap \Pi$.
 The lines $l_1$ and $l_2$ are supporting in $\Pi$ the set $\mathcal C \cap \Pi$ at the points $A^*$ and $B^*$ , respectively, and $[A^*B^*]$ is a local minimizer for the set $\mathcal C \cap \Pi$  in $\Pi$.
 Proposition \ref{dim2prop} gives us three concurrent lines $p_0, p_1, p_2$ in $\Pi$ (with common point $P$)
 which are perpendicular , respectively, to the line $A^*B^*$ at $O$, to $l_1$ at $A^*$ and to $l_2$ at $B^*$.
 As $p_1$ lies in $\Pi$, it is orthogonal to $ T + L$. It is also perpendicular to $l_1$.
This implies that $p_1$ coincides with the normal line to $\Pi_1$ at $A^*$.
 Similarly, $p_2$ coincides with the normal line to $\Pi_2$ at $B^*$. Clearly, the line $p_0$, together with its point $P$ lies in the hyperplane through $O$ which is orthogonal to the line $A^*B^*$.
 Thus (ii)  holds for any nonparallel hyperplanes $\Pi_1$ and $\Pi_2$ supporting $\mathcal C$ at $A^*$ and $B^*$ , respectively.

 \smallskip

 Note that the line $p_1$, the normal to $\Pi_1$ at $A^*$, does not depend on the hyperplane $\Pi_2$ supporting $\mathcal C$ at $B^*$.
 Hence, the point $P$, as the intersection of $p_1$ and the hyperplane through $O$
 orthogonal to the line $A^*B^*$ is the same for every hyperplane $\Pi_2$ supporting $\mathcal C$ at $B^*$ and
 nonparallel to $\Pi_1$. This, in turn, means that any hyperplane $\Pi_2'$ supporting $\mathcal C$ at $B^*$ and
 not parallel to $\Pi_1$ is orthogonal to the line $PB^*$ and, therefore coincides with $\Pi_2$.

 In order to complete the proof, we have to discuss also the case where, along with $\Pi_2$,  there is another hyperplane $\Pi'$
 which supports $\mathcal C$ at $B^*$ and is parallel to $\Pi_1$.
 Suppose such a hyperplane $\Pi'$ exists. In this case, there is a linear subspace $L \subset \mathbb R^n$ of codimension one such that $\Pi_1 = A^* + L$ and $\Pi' = B^* + L$.
As above we can see that for every non zero vector $h \in L$ the
line $A^*B^*$ is perpendicular to the lines $l_1(h)=\{A^* + th: t
\in \mathbb{R}\}$ and $l_2(h)=\{B^* + th: t \in \mathbb{R}\}$.
This means that $A^*B^*$ is normal to both $\Pi_1$ and $\Pi'$ at
$A^*$ and $B^*$ , respectively. Thus $p_1 = A^*B^* $ and,
therefore, $P \in A^*B^*$. Since $PB^*$ is orthogonal to $\Pi_2$
we get $p_2 = A^*B^*$. This contradicts the assumption that
$\Pi_1$ and $\Pi_2$ are not parallel.
 \end{proof}

 As we have seen in the previous section (Proposition \ref{polygon}) one cannot expect a similar result for local maximizers.
 If, however, the set $\mathcal C$ is smooth, then CPP holds for local maximizers as well.
 \begin{thm}
 Let $O$ be an interior point for a smooth closed convex set $\mathcal D \subset \mathbb{R}^n,\, n \geq 2$, and let
 the segment $[A^*B^*]$ be a local maximizer for $F$. Let $\Pi_1$ and $\Pi_2$ be the only supporting hyperplanes at $A^*$ and $B^*$, respectively. Then
the normal line to $\Pi_1$ at $A^*$ and the normal line to $\Pi_2$
at $B^*$ intersect at a point belonging to the hyperplane through $O$ which
is orthogonal to the line $A^*B^*$.
 \end{thm}
The proof of this theorem is omitted. It almost coincides with the proof of Theorem \ref{thmmin}. One uses Lemma \ref{observe}
to show that the line$A^*B^*$ is perpendicular to the parallel lines $l_1(h)$ and $l_2(h)$. Instead of  Proposition \ref{dim2prop} one uses now Proposition \ref{max}. The possibility to apply Proposition \ref{max} is based on the following simple observation, which is easily proved by a separation theorem argument.

\begin{lemma} \label{relative}
Let $\mathcal D$ be a smooth closed convex subset of
$\mathbb{R}^n,\, n > 2$, with nonempty interior, and let a two
dimensional plane $\Pi$ intersects its interior. Then the
intersection $\mathcal C':=\mathcal C \cap \Pi$ is a smooth convex
subset of $\Pi$ (i.e. there is a unique supporting for $\mathcal
C'$ line  in $\Pi$ at every point of the boundary $\partial \,
\mathcal C'$ in  $\Pi$).
\end{lemma}

Minor changes are needed to formulate the corresponding results for the case when the point $O$ is outside the set $\mathcal C$.

Let $\mathcal C \subset \mathbb{R}^n,\, n \geq 2,$ be a closed convex set  with nonempty interior and let  $O$ be a point outside it.
Consider all segments $[AB]$
 with end-points $A$ and $B$ on the boundary $\partial \, \mathcal C$, and such that the line $AB$ passes through $O$ and $B \in [OA]$.
 Define for such segments the  function  $G([AB])=|AB|$. The definition of ``local minimizer''and ``local maximizer''for $G$ are introduced as in Definition \ref{defi-min-max}.

\begin{thm} \label{Ooutsidemin}
Let $O$ be a point outside a closed convex set $\mathcal C \subset
\mathbb{R}^n,\, n \geq 2$, with nonempty interior, and let the
segment $[A^*B^*]$ be a local minimizer for $G$. Suppose that
$[A^*B^*]$ intersects the interior of $\mathcal C$. Then:

(i) There exists only one hyperplane $\Pi_1$   supporting $\mathcal C$   at  $A^*$ and only one hyperplane $\Pi_2$ supporting $\mathcal C$  at $B^*$;

(ii) The normal line to $\Pi_1$ at $A^*$ and the normal line to
$\Pi_2$ at $B^*$ intersect at a point in the hyperplane through
$O$ which is orthogonal to the line $A^*B^*$.
\end{thm}
The proof of this theorem follows step-by-step the proof of Theorem \ref{thmmin}. Instead of Proposition \ref{dim2prop} one uses Proposition \ref{tilde O outside}.

\begin{thm}
Let $O$ be a point outside a smooth closed convex set $\mathcal D
\subset \mathbb{R}^n,\, n \geq 2$ with nonempty interior and
suppose the segment $[A^*B^*]$ is a local maximizer for $G$ which
intersects the interior of $\mathcal D$. Let $\Pi_1$ and $\Pi_2$
be the only supporting for $\mathcal D$ hyperplanes at $A^*$ and
$B^*$, respectively. Then the normal line to $\Pi_1$ at $A^*$ and
the normal line to $\Pi_2$ at $B^*$ intersect at a point belonging to the
hyperplane through $O$ which is orthogonal to the line $A^*B^*$.
\end{thm}

The proof of this theorem follows the steps in the proof of Theorem \ref{thmmin}.
One uses Lemma \ref{observe1}
to show that the line$A^*B^*$ is perpendicular to the parallel lines $l_1(h)$ and $l_2(h)$.
Instead of using Proposition \ref{dim2prop}
one can apply Lemma \ref{relative} and Proposition \ref{max O outside}.

\begin{rem} {\rm
Using the method of Lagrange multipliers, one may show that the theorems proved above for convex sets in $\mathbb R^n, n\geq 2$, remain true for any $C^1$-smooth hypersurface in $\mathbb R^n$ (allowing the intersection point of the normals to be an infinite point). Details will appear in a forthcoming paper.}
\end{rem}

\medskip
At the end of this section, we present
a generalization of
Proposition \ref{polygon} which was proved for polygons in $\mathbb R^2$. The generalization we present is valid for arbitrary closed convex sets $\mathcal C \subset \mathbb R^n$. This generalization will be used in the next section.

\begin{defi} \label{dimlem}
{\rm Let  $\mathcal C \subset \mathbb{R}^n,\, n \geq 2,$ be a closed convex set with nonempty interior
 and $A \in \mathcal C$. Under a {\it dimension of $A$ with respect to $\mathcal C$} (denoted by $d_{\mathcal C}(A)$) we understand the maximal integer $k$
 for which there exists a $k$-dimensional affine subspace $H \subset \mathbb {R}^n$  such that $A$ belongs to the interior
of the set $H \cap \mathcal C$ in $H$.}
\end{defi}
In particular, $A$ is an interior point
of $\mathcal C$
if and only if $d_{\mathcal C}(A) = n.$ Also, $A$ is an extreme point of $\mathcal C$ if and only if $d_{\mathcal C}(A)=0$.
The most important case for us is when
$A \in \partial \, \mathcal C$. In this case,
the affine subspace $H$ in the above definition does not intersect the interior of $\mathcal C$ . Further, it is easy to realize that $H$ is uniquely determined.

\begin{prop} \label{p5}
Let  $\mathcal C \subset \mathbb{R}^n,\, n \geq 2,$ be a closed convex set with nonempty interior
 and $O$ a point in $\mathbb{R}^n $. Consider all lines through $O$ which intersect $\mathcal C$ in some (possibly degenerated) segment $[AB]$.

 (i) if $O \not \in \mathcal C$ and the segemnt
 $[A^*B^*]$ is a local maximizer for the function $G([AB])$ defined above, then
  $$d_{\mathcal C}(A^*) + d_{\mathcal C}(B^*) \leq n.$$

 (ii) if $O$ is an interior point of $\mathcal C$ and  the segment
 $[A^*B^*]$ is a local maximizer
for the function $F([AB])$ defined above, then

 $$d_{\mathcal C}(A^*) + d_{\mathcal C}(B^*)  \leq n-1.$$

\end{prop}

\begin{proof}
First note that in both cases $A^* \not = B^*$. For the case (ii) this is obvious because $O$ is an interior point for
$[A^*B^*]$. In the case (i), this follows from the assumptions that $\mathcal C$ has nonempty interior and that $[A^*B^*]$ is a local maximizer.
Let $H_1, H_2$ be the affine subspaces from
the definition of $d_{\mathcal C}(A^*)$ and $d_{\mathcal C}(B^*)$ respectively.
Since $A^*$ is an interior point (relative to $H_1$) of the set $H_1 \cap \mathcal C$,
we derive that $B^* \not \in  H_1$ (otherwise there would exist a poit $A \in \mathcal C$, different from
$A^* $, and such that $A^* \in [AB^*]$; this would contradict the local maximality of
$[A^*B^*]$).
Similarly, one derives that  $H_2$ does not contain $A^*$.
Further, there are linear
 subspaces $L_1, L_2$ such that $H_1 = A^* +L_1$ and $H_2 = B^* +L_2$.We show next that the
intersection $L_1 \cap L_2$ contains only the zero element of $\mathbb R^n$. Indeed, suppose
there exists a non-zero vector $h \in L_1 \cap L_2$.
Consider the two-dimensional plane
$\Pi$ containing the strip bounded by the two different parallel lines
$\{A^* + th : t\in \mathbb R\} \subset H_1$  and
$\{B^* + th : t\in \mathbb R\} \subset H_2 $.
$\Pi$ contains the line $A^*B^*$ and, therefore, the point $O$ as well.
As $A^*$ is from the interior of $\mathcal C \cap H_1$ relative to $H_1$,
there ezists some $t_1>0$ such that $[A^* - t_1h, A^* + t_1h] \subset \mathcal C$. Similarly, there exists $t_2 >0$ such that  $[B^* - t_2h, B^* + t_2h] \subset \mathcal C$.
This means that the segment
$[A^*B^*]$ is a local maximizer for the function $g$ from Section 2 corresponding to this strip when $O \not \in \mathcal C$ and local maximizer for the function $f$ from Section 2 corresponding to this strip when
$O$ is from the interior of $\mathcal C$.
This contradicts the strict convexity of the functions $f$ and $g$.
 Therefore, both in the case (i) and in the case (ii) we have
$$d_{\mathcal C}(A^*) + d_{\mathcal C}(B^*) = \dim L_1 + \dim L_2 = \dim(L_1 + L_2) \leq n.$$ This completes the proof of (i).

Let now $O$ be an interior point of $ \mathcal C$. We prove first that $H_1 \cap H_2 = \emptyset$. Indeed, if there is a point $E \in H_1 \cap H_2$, then $E\not = A^*$ and $E \not= B^*$. Put $v_1= \overrightarrow{A^*E}$ and
$v_2 = \overrightarrow{B^*E}$. Since $A^*$ is in the relative interior of
$\mathcal C \cap H_1$in $H_1$, there is some $r_1>0$ such that
$[A^* - r_1v_1, A^* + r_1v_1] \subset \mathcal C$.
Similarly, there exists $r_2 >0$ such that
$[B^* - r_2v_2, B^* + r_2v_2] \subset \mathcal C$.
 Consider the two-dimensional plane $\Pi'$ containing the angle $\angle A^*EB^*$. The segment $[A^*B^*]$ being a local maximizer for the function $F$ will be a local maximizer for the function $f$ from Section 2 corresponding to this angle. This contradicts the strict convexity of the function $f$.

It is now easy to realize that $L_1 + L_2$ does not contain the vector $\overrightarrow{A^*B^*}$. Indeed, if
$\overrightarrow{A^*B^*}= v_1 + v_2$
where $v_1 \in L_1$ and $v_2 \in L_2$ then $E:=A^* +v_1=B^* - v_2 \in H_1 \cap H_2$, This implies
$$d_{\mathcal C}(A^*) + d_{\mathcal C}(B^*) = \dim L_1 + \dim L_2 = \dim(L_1 + L_2) < n.$$
\end{proof}

\begin{cor}
Let  $\mathcal C \subset \mathbb{R}^n,\, n \geq 2,$ be a closed convex set
and $O$ be an interior point of it.
Let $[A^*B^*]$ be a local maximizer for the function $F([AB])$ and let there exist a hyperplane $H_{A^*}$ supporting
$\mathcal C$ at $A^*$ and such that $A^*$ is an interior point
(relative to $H_{A^*}$) of the set
$\mathcal C \cap H_{A^*}$.
Then $B^*$ is an extreme point of $\mathcal C$.
 \end{cor}
 \begin{proof}
 Since $d_{\mathcal C}(A^*)= n-1$, we derive that $d_{\mathcal C}(B^*)=0$.
 \end{proof}

 Similarly, If $d_{\mathcal C}(A^*) = n-2$, then $B^*$ is either an extreme point
 of $\mathcal C$ or lies on some segment contained in $\partial \, \mathcal C$.
 Note also that in the case when $\mathcal C$ is a polygon in $\mathbb R^2$, this corollary reduces to Proposition \ref{polygon}.

\section{Convex polytopes}\label{sc}

This section presents some results for compact convex sets $\mathcal C \subset \mathbb R^n, n\geq 2, $ with nonempty interior which are polytopes (convex hulls
of finitely many points or, equivalently, bounded sets which can be presented as the
 intersections of finitely many closed half-spaces determined by hyperplanes).
 \smallskip

As was mentioned above, the planar version of Proposition \ref{p5} generalizes Proposition
\ref{polygon}: if $[A^*B^*]$ is a local maximizer for a convex polygon $\mathcal C$
and $O$ is a point from its interior, then at least one of the points $A^*$ and $B^*$ is a vertex of the polygon.
The following example of a tetrahedron shows that this is not true in higher dimensions.

\noindent{\bf Example 1.} Let $\mathcal C$ be the convex hull of the points $A_1(1,-1,0)$, $A_2(1,1,0), A_3(-1,0,1)$, $A_4(-1,0,-1)$. Let $O$ be the point $(0,0,0)$ (Figure 6).

 \begin{figure}[h]  \label{figure 6}
\centering
\includegraphics[scale=0.8]{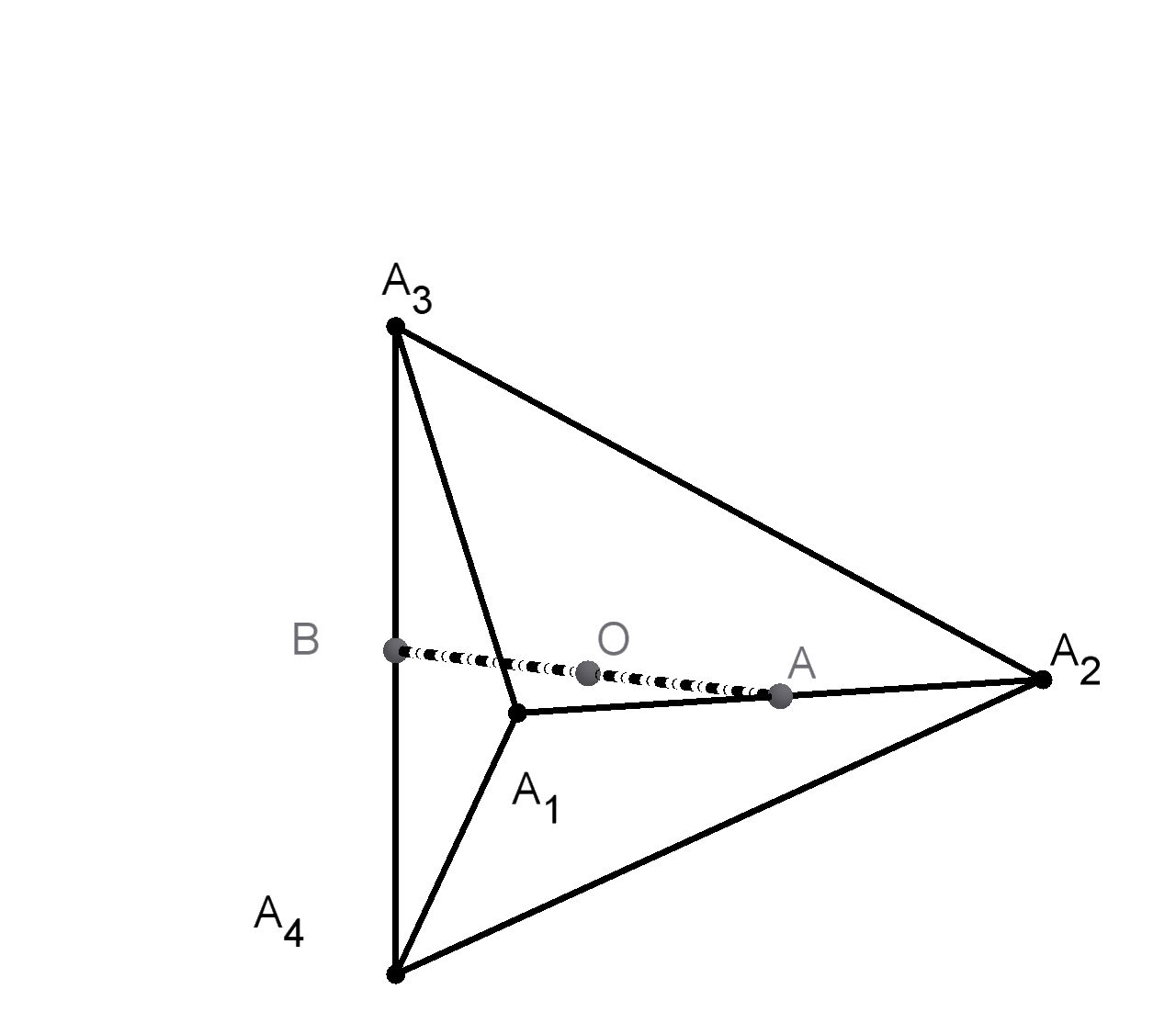}
\caption*{Figure 6}
\end{figure}

Every line through $O$ cuts a non-degenerate segment from $\mathcal C$. Consider the segment $[AB]$ with end-points $A(1,0,0)$ and $B(-1,0,0)$.
Then it is easy to check that the four segments cut from the tetrahedron by lines $OA_i$, $i=1,2,3,4$
have the same length
equal to $\displaystyle\frac{4\sqrt{5}}{5}<2=|AB|.$ So, the
endpoints of any global maximizer $[A^*B^*]$ for the function $F$ are not among the vertices of the tetrahedron.
Proposition \ref{p5} (ii) implies now that the end points of $[A^*B^*]$ are also not internal points of the triangles on the surface of the
tetrahedron. The only remaining options for
$A^*$ and $B^*$ are to belong to the opposite edges $[A_1A_2]$ and $[A_3A_4]$ (as does $[AB]$) or to belong to a pair of opposite skew edges.
For symmetry reasons, and because $O \in [A^*B^*]$, the points $A^*$ and $B^*$ must be the mid-points of the
corresponding opposite skew edges.
There are only three such segments, and they are of equal length. Let the segment $[CD]$ where $C=(0, -0.5, -0.5)$ (middle of $[A_1A_4]$)
and $D=(0, 0.5, 0.5)$ (middle of $[A_2A_3]$) be one of these three segments.
Its length is $\sqrt{2} <2$. All this shows that $[AB]$ is the only global maximizer $[A^*B^*]$ for the function $F$ corresponding to the tetrahedron $\mathcal C$ and the point $O$.

\medskip

Let us consider another example which shows that Proposition \ref{polygon} is not valid if $O$ is outside the polygon.

\noindent{\bf Example 2.} Let $\mathcal C \subset \mathbb R^2$ be the right-angled triangle $PQR$ where $P=(0,0)$, $Q=(6,0)$, $R=(0,2)$, and let $O=(0,3)$ (Figure 7). Consider the line $OA$ where $A=(3,0)$. Then $B=(1.5, 1.5)$
and $|AB| = \sqrt{4.50} >2 = |PR|$. This suffices to conclude that the end-points of a global
maximizer $[A^*B^*]$ for the function $G$ corresponding this problem are not among vertices of $\mathcal C$.

 \begin{figure}[h]  \label{figure 7}
\centering
\includegraphics[scale=0.8]{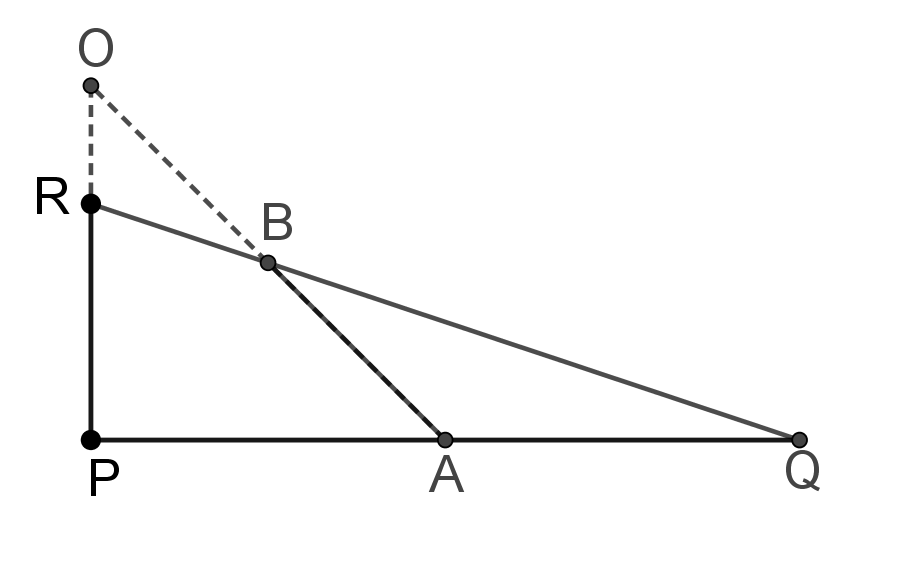}
\caption*{Figure 7}
\end{figure}

Our next goal is to show that, when the point $O$ is far from a polygon
$\mathcal C \subset \mathbb R^2$, then the local maximizers do contain a vertex of it.
More precisely, we will show that the
inequality in Proposition \ref{p5} (ii) is valid also in the case when $O$ is at
a sufficiently large distance from the convex polytope $\mathcal C \subset \mathbb R^n, n \geq 2$.
\smallskip

We use the term {\it ``exposed face of the polytope $\mathcal C$''} for any set $\mathcal C \cap \Pi$
where $\Pi$ is a supporting hyperplane for $\mathcal C$. An exposed face $\mathcal C \cap \Pi$ that
has nonempty interior relative to $\Pi$ is called {\it ``facet''}.
By {\it ``dimension of an exposed face''} we have in mind the dimension of the affine hull of the exposed face. This affine hull is a translation of some linear subspace $L$ of the same dimension.
In particular, the dimension of a facet is $n-1$.
Exposed faces of dimension 1 are also called {\it ``edges''}.
The vertices of the polytope are exposed faces of dimension 0.

Let further $A$ be a point from the boundary of the polytope $\mathcal C$ and $H$ be the
affine subspace from Definition \ref{dimlem} which determines the dimension $d_{\mathcal C}(A)$. It is easy to see that $H$ does not intersect the interior of $\mathcal C$.  Also, it is possible to prove
that $H \cap \mathcal C$ is an exposed face of the polytope.
Finally, it is well-known that any polytope has finitely many exposed faces.

\begin{prop}\label{p8}
For any convex polytope $\mathcal C$ in $\mathbb R^n$  with nonempty interior
there exists a bounded set $U$ such that if $O\not\in U,$ and $[A^*B^*]$ is a local
maximizer for the function $G$
then
$$d_{\mathcal C}(A^*)+d_{\mathcal C}(B^*) \le n-1.$$
In particular, when $n=2$, at least one of the points $A^*$ and $B^*$ is a vertex of the polygon $\mathcal C$.
\end{prop}

\begin{proof} We start with the following

\begin{lemma}\label{l9} Let $A, B, O, E$ be points in the plane $\mathbb R^2$ such that: $O$ is
outside the angle $\angle AEB$, $O \in AB$, $B \in [AO]$, and
$0 < \theta :=\angle AEB < \pi/2$ (Figure 8).
Suppose CPP has a place for the points $A, B, O$ and the angle $\angle AEB$. Then

$$\frac{2|OA|}{|AB|}\le 1+\frac{1}{\sin\theta}.$$
\end{lemma}

 \begin{figure}[h]  \label{figure 8}
\centering
\includegraphics[scale=0.4]{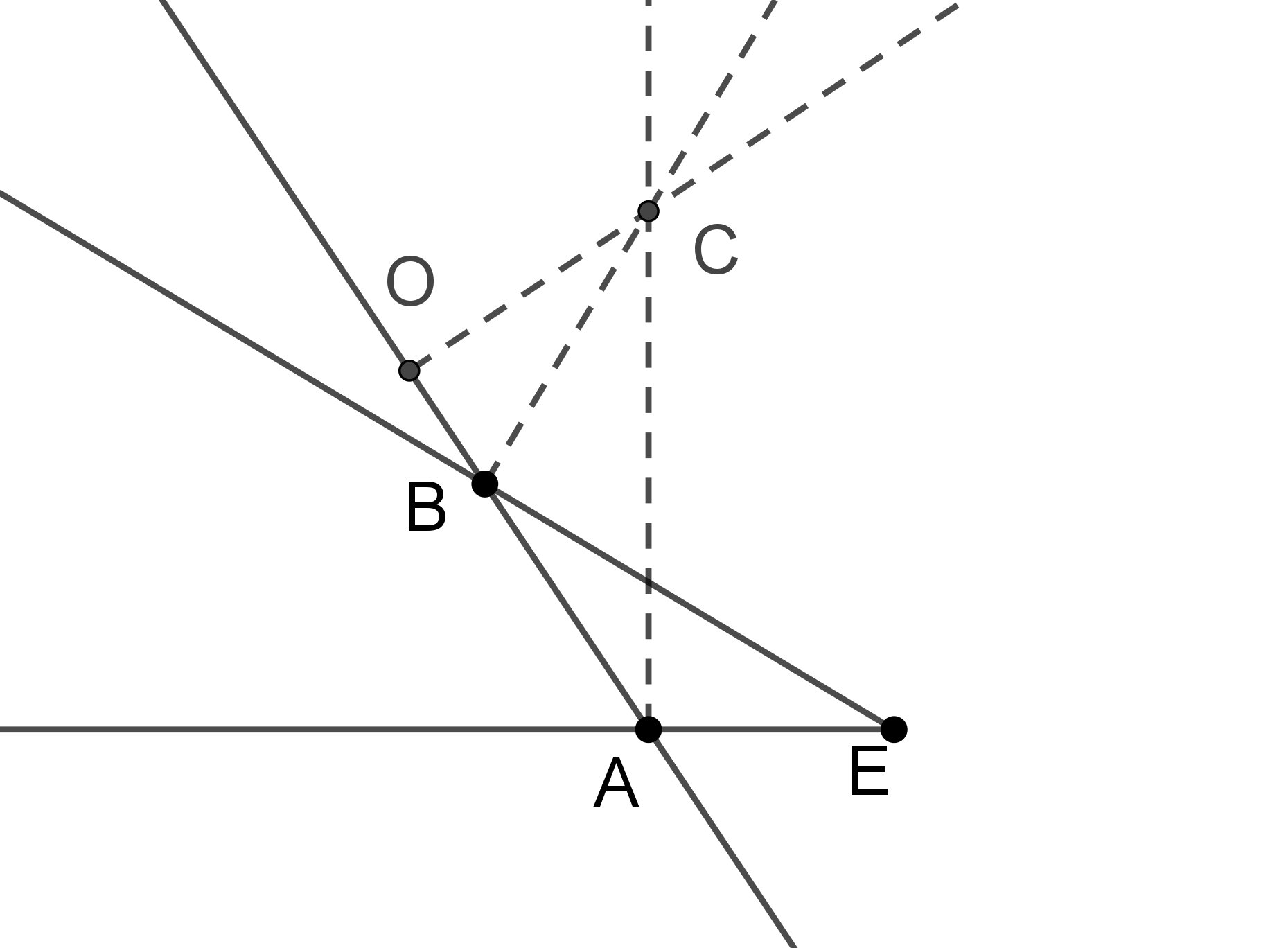}
\caption*{Figure 8}
\end{figure}

\noindent{\it Proof of the lemma.} Denote by $C$ the
intersection point of the perpendiculars to $EA, EB$ and $AB$ at
$A,B$ and $O$, respectively. Then $EABC$ is a convex quadrilateral
inscribed in the circle with diameter $EC.$ Set  $\psi=\angle BAC.$ Then
$$AB=EC\sin\theta,\quad OA=AC\cos\psi,\quad AC=EC\sin(\theta+\psi)$$
and hence
$$\frac{2OA}{AB}=\frac{2\sin(\theta+\psi)\cos\psi}{\sin\theta}
=1+\frac{\sin(\theta+2\psi)}{\sin\theta}$$ which implies the
desired inequality and completes the proof of the lemma.
\medskip

Further, let $\mathcal L$ be the family of all pairs $(L_1,L_2)$ of
linear subspaces of $\mathbb R^n$, each being a translation of the affine hull of some exposed face of $\mathcal C$, and such that
$L_1 \cap L_2$ has just one element (the zero elementt of $\mathbb R^n$).  $\mathcal L$ is
 a nonempty family (each  vertex of $\mathcal C$ belongs to at least two different edges to which correspond different one-dimensional linear subspaces $L_1$ and $L_2$). $\mathcal L$ is also a finite family.
Put
$$c=\max_{(L_1,L_2)\in\mathcal L}\max\{\langle u_1,u_2\rangle:
u_1\in L_1, u_2\in L_2, ||u_1||=||u_2||=1\}.$$
Compactness of the unit sphere and continuity of the function $\langle u_1,u_2\rangle$ yield that $c\in[0,1).$

Suppose now $[A^*B^*]$ is a local maximizer  for the function $G$ which correspond to the polytope $\mathcal C$ and a point
$O\not\in \mathcal C$.
Let $H_{A^*}, H_{B^*}$ be the affine subspaces from the definition of $d_{\mathcal C}(A^*)$
and $d_{\mathcal C}(B^*)$ respectively.
Denote by $L_{A^*}$ and $L_{B^*}$ the linear
spaces such that $H_{A^*} = A^* + L_{A^*}$
and $H_{B^*} = B^* + L_{B^*}$.
From the proof of Proposition \ref{p5} we know that $A^* \not = B^*$ and that $L_{A^*} \cap L_{B^*}$ contains only the zero vector of the
space. Hence the pair $(L_{A^*}, L_{B^*}) $ belongs to $\mathcal L$. Proposition \ref{p5} (i) says also that $\dim(L_{A^*}) + \dim(L_{B^*}) \leq n$.
\smallskip

Let $U :=\{X\in\mathbb R^n: \text{dist}(X,
\mathcal C)\le M\text{diam}(\mathcal C)\}$ where $M:=\frac{1}{2}\left(1+\frac{1}{\sqrt{1-c^2}}\right)$, and let $O \not \in U$. Suppose
$d_{\mathcal C}(A^*) + d_{\mathcal C}(B^*)=\dim(L_{A^*}) +\dim(L_{B^*}) = n$.  Then the non-zero vector $\overrightarrow{A^*B^*} = h' + h''$ where $h'\in L_{A^*}$ and $h''\in L_{B^*}$.
Put $E= A^* + h'= B^* - h''$ and consider
the angle $\angle A^*EB^*$. Note that $A^*$ belongs to the relative interior of
$H_{A^*}\cap \mathcal C$ in $H_{A^*}$ and
$B^*$ belongs to the relative interior of
$H_{B^*}\cap \mathcal C$ in $H_{B^*}$. Therefore, the segment $[A^*B^*]$ is a local
maximizer for the function $g$ corresponding to the angle $\angle A^*EB^*$ and the point $O$.
This implies (see Remark \ref{ostar}) that
$\angle A^*EB^*<\pi/2$. Then Lemma \ref{l9} implies that $O \in U$, a contradiction.

\end{proof}

At the end of this section, we show that, for a large class of polytopes, the function $G$ does not have
minimizers intersecting the interior of $\mathcal C$ when $O$ is far away from the polytope.

\begin {prop}
Let $\mathcal C$ be a polytope with nonempty interior and without parallel facets. Then there exists a bounded set $U$ such that if $O\not\in U,$ the function $G$ does not have local minimizers intersecting the interior of $\mathcal C$. In particular, every local minimizer of $G$ is contained in the boundary of $\mathcal C$ when $O$ is at a large distance from $\mathcal C$.
\end{prop}

\begin{proof}
Let $O \not \in \mathcal C$ and $[A^*B^*]$ be a local minimizer for the function $G$ which contains interior points of $\mathcal C$.
According to Theorem \ref{Ooutsidemin}, only one hyperplane supports the polytope at the point $A^*$. This means that
$A^*$ is an
interior point of the facet determined by this hyperplane and, therefore, this hyperplane coincides with $H_{A^*}$.
Similarly,  $B^*$ is an
interior point of the facet $\mathcal C \cap H_{B^*}$ relative to $H_{B^*}$.
Put $H:= H_{A^*} \cap H_{B^*}$. Since there are no parallel facets, $H$ is a nonempty $(n-2)$-dimensional affine subspase. Let $L$ be the $(n-2)$-dimensional linear subspace corresponding to $H$.
Take a nonzero vector $h \in L$ and consider the strip
bounded by the two dufferent lines $\{A^* + t h : t \in R\} \subset H_{A^*}$
and $\{B^* + t h : t \in R\} \subset H_{B^*}$. Then
$[A^*B^*]$ is a local minimizer for the function $g$ corresponding to this strip. This implies that the line $OA^* = OB^*$ is perpendicular to the vector $h$ and, therefore, to $L$.
Hence, there exists a uniquely determined two-dimensional plane $\Pi'$
that contains $OA^*$ and is orthogonal to the
$(n-2)$-dimensional affine subspace $H$.
Let $E= H \cap \Pi'$. Then $[A^*B^*]$ is a local minimizer for the function $g$
corresponding to the angle $\angle A^*EB^*$.
Note that the angle $\angle A^*EB^*$ is, actually, the angle between the hyperplanes $H_{A^*}$ and $H_{B^*}$, and we consider it to be the angle between the two facets $H_{A^*}\cap \mathcal C$ and $H_{B^*}\cap \mathcal C$.
This implies (see Remark \ref{ostar}) that $\angle A^*EB^* <\pi/2$.
By Lemma \ref{l9}, we get
$$\frac{2|OA^*|}{|A^*B^*|}\le 1+\frac{1}{\sin \angle  A^* E B^*}.$$
 Denote by $\Theta$ the set of all numbers $\theta$ which are angles between two facets of $\mathcal C$.
 Since there are no parallel facets, and the number of facets is finite, the number
 $m= min\{|sin \theta|: \theta \in \Theta\} >0$. It follows that $O \in U$ where
$$U :=\{X\in\mathbb R^n: \text{dist}(X,
\mathcal C)\le M\text{diam}(\mathcal C)\}\, \text{with}\, M:=\frac{1}{2}\left(1+\frac{1}{m}\right).$$

\end{proof}

\section{Appendix}
This section presents briefly two practice-oriented problems in which CPP plays some role.

\subsection {A special Parking problem (\cite{PSK}).}

A car (wheel-chair, baby carriage) is on the street and has to be parked in the basement of a house over a slope (see Figure 9) without damaging the bottom of the vehicle.

\begin{figure}[h]  \label{f9}
\centering
\includegraphics[scale=0.7]{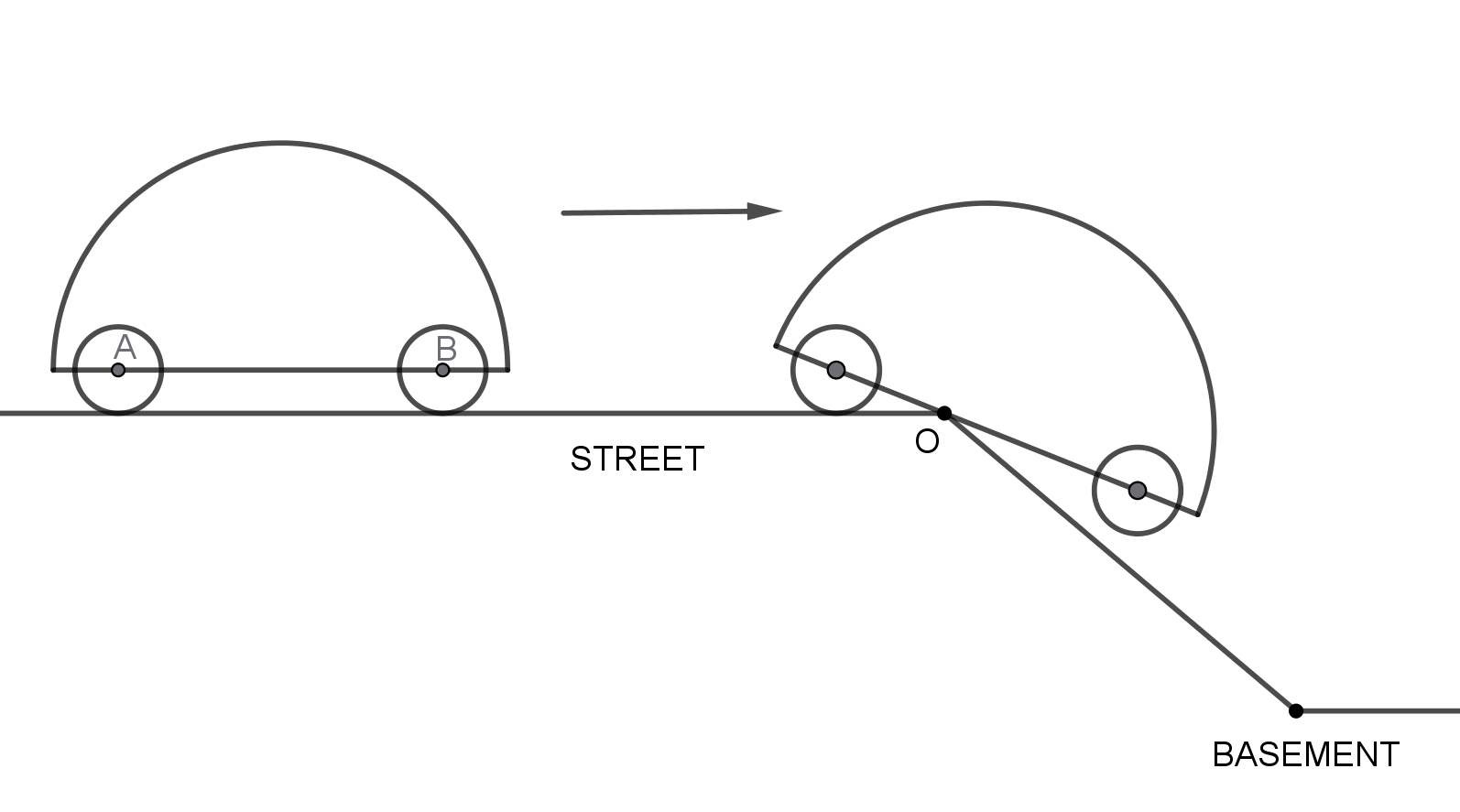}
\caption*{Figure 9}
\end{figure}

Whether or not this is possible depends on the distance $d $ between the centers of the wheels (points $A$ and $B$), the radius  $r$ of the wheels and the steepness of the slope. Suppose the slope and the radius $r$ are given. Then, a natural question arises:

{\it ``What is the maximal $d$ such that a car with $d = |AB|$ can overcome the slope keeping the bottom safe?''}

 An upper bound for $d$ is given by the mathematical problem considered in this paper. It is depicted in Figure 10:

{\it ``Find the shortest segment $[DE]$ passing through $O$ with end points on the dotted line (consisting of points $r$-distanced from the street or the slope)?''.}

\begin{figure}[h]  \label{f9}
\centering
\includegraphics[scale=0.8]{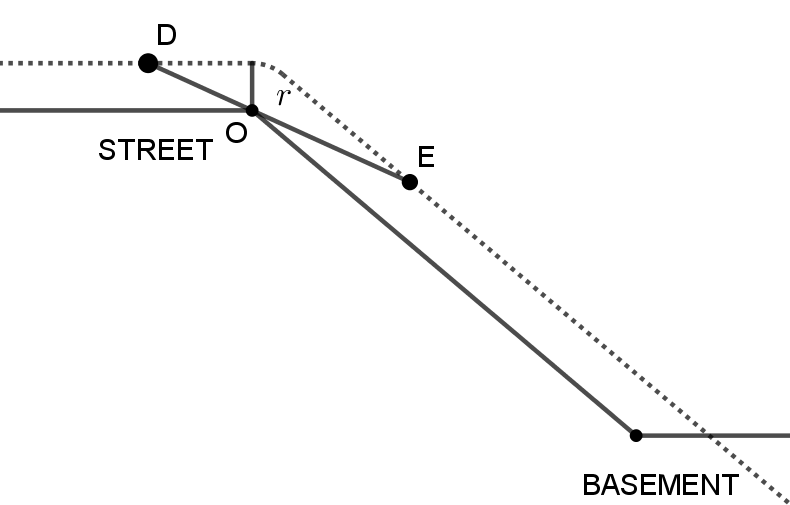}
\caption*{Figure 10}
\end{figure}

  As we know, the solution to this problem is governed by CPP.

 \subsection {Estimating the distance to a broadcasting object}

 A point-like source $O$ casts (radiates, sends) signals in all directions. An observer knows the direction the signals are coming from and wants to find, at least approximately, the distance to the source.
The observer disposes with a solid object (shaped as in Figure 11) made of material which absorbs the intensity (the strength) of the signals proportionally to the path-length the signals travel through the object.

\begin{figure}[h]  \label{f8}
\centering
\includegraphics[scale=0.3
]{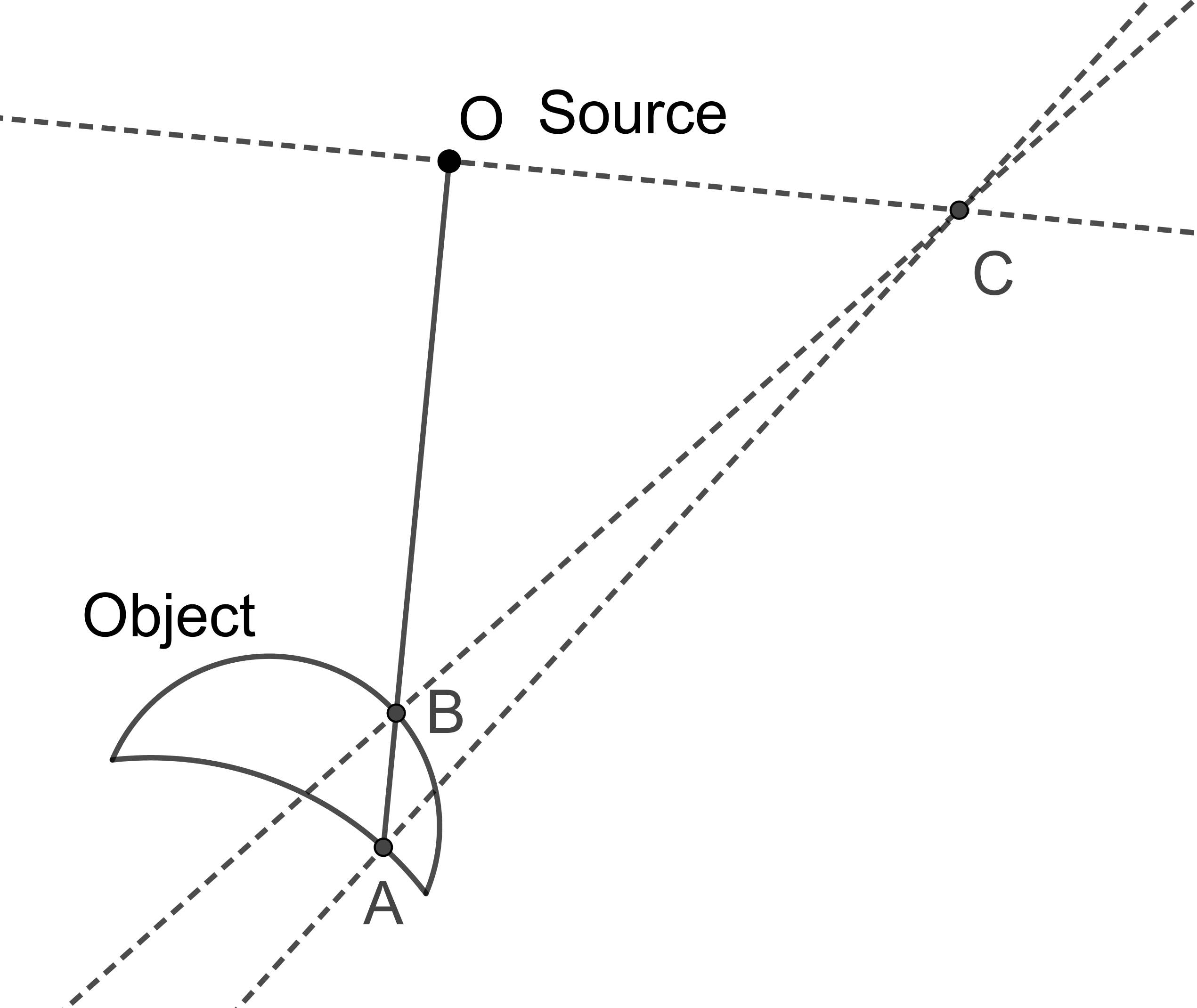}
\caption*{Figure 11}
\end{figure}

Placing sensors along the smaller arc, one finds the point $A$
where the signal intensity is minimal. Since the direction to the
source $O$ is known, one finds the point $B \in [OA]$ on the other
arc of the object. The segment $[AB]$ is a local maximum among the
segments cut from the object by lines containing point $O$. Since
both arcs are smooth curves, Lemma \ref{linearization} implies
that CPP holds. Let $C$ be the intersection of the normals to the
object at the points $A$ and $B$. By CPP, the origin $O$ of the
signals must belong to the perpendicular from $C$ to $AB$.

{}

\end{document}